\documentclass[11pt,a4paper]{article}

\usepackage{amsmath}
\usepackage{amssymb}
\usepackage{mathtools}
\usepackage{hyperref}
\makeatletter
\usepackage[margin=1in]{geometry}
\usepackage{amsfonts}
\usepackage{amssymb}\usepackage{amsmath}
\usepackage{amsfonts}
\usepackage{amssymb}
\usepackage{amsthm}
\usepackage{fixltx2e}
\usepackage{tikz}
\usepackage{tikz-cd}
\usetikzlibrary{decorations.markings}
\newtheorem{theorem}{Theorem}[section]

\newtheorem{proposition}[theorem]{Proposition}
\newtheorem{corollary}[theorem]{Corollary}
\theoremstyle{definition}
\newtheorem{definition}{Definition}[section]

\theoremstyle{remark}
\newtheorem{remark}{Remark}
\AtEndDocument{
\textsc{Mathematical and Statistical Sciences, University of Alberta, Edmonton, Canada} \\
\textsc{Email:} \texttt{primozic@ualberta.ca}}
\begin{document}
\title{Computations of de Rham cohomology rings of classifying stacks at torsion
primes}
\author{Eric Primozic}
\maketitle
\setcounter{secnumdepth}{1}
\section*{Introduction}
\addcontentsline{toc}{section}{Introduction}
\ 
\
\indent Let $G$ be a smooth affine algebraic group over a commutative ring $R$. In \cite{Tot}, Totaro defines the Hodge cohomology group $H^{i}(BG, \Omega^{j})$ for $i, j \geq 0$ to be the $i$th \'etale cohomology group of the sheaf of differential forms $\Omega^{j}$ over $R$ on the big \'etale site of the classifying stack $BG$. For $n \geq 0$, let $H^{n}_{\textup{H}}(BG/R) \coloneqq \oplus _{j} H^{j}(BG, \Omega^{n-j})$ denote the total Hodge cohomology group of degree $n$. De Rham cohomology groups $H^{n}_{\textup{dR}}(BG/R)$ are defined to be the \'etale cohomology groups of the de Rham complex of $BG$. Let $\frak{g}$ denote the Lie algebra associated to $G$ and let $O(\frak{g})=S(\frak{g}^{*})$ denote the ring of polynomial functions on $\frak{g}$. In \cite[Corollary 2.2]{Tot}, Totaro showed that the Hodge cohomology of $BG$ is related to the representation theory of $G$:
$$H^{i}(BG, \Omega^{j}) \cong H^{i-j}(G, S^{j}(\frak{g}^{*})).$$  

\indent Let $G$ be a split reductive group defined over $\mathbb{Z}$. From the work of Bhatt-Morrow-Scholze in p-adic Hodge theory \cite[Theorem 1.1]{BMS}, one might expect that 

\begin{equation} \label{comparison special fiber and generic} 
\textup{dim}_{\mathbb{F}_{p}} H^{i}_{\textup{dR}}(BG_{\mathbb{F}_{p}}/\mathbb{F}_{p}) \geq \textup{dim}_{\mathbb{F}_{p}}   H^{i}(BG_{\mathbb{C}}, \mathbb{F}_{p}) 
\end{equation} 
for all primes $p$ and $i \geq 0$. The results from \cite{BMS} do not immediately apply to $BG$ since $BG$ is not proper as a stack over $\mathbb{Z}$. For $p$ a non-torsion prime of a split reductive group $G$ defined over $\mathbb{Z}$, Totaro showed that \begin{equation} \label{nontorsion}H^{*}_{\textup{dR}}(BG_{\mathbb{F}_{p}}/\mathbb{F}_{p}) \cong H^{*}(BG_{\mathbb{C}}, \mathbb{F}_{p})\end{equation} \cite[Theorem 9.2]{Tot}. It remains to compare $H^{*}_{\textup{dR}}(BG_{\mathbb{F}_{p}}/\mathbb{F}_{p})$ with $H^{*}(BG_{\mathbb{C}}, \mathbb{F}_{p})$ for $p$ a torsion prime of $G$. For $n \geq 3$, $2$ is a torsion prime for the split group $SO(n)$. Totaro showed that $$H^{*}_{\textup{dR}}(BSO(n)_{\mathbb{F}_{2}}/\mathbb{F}_{2}) \cong H^{*}(BSO(n)_{\mathbb{C}}, \mathbb{F}_{2}) \cong \mathbb{F}_{2}[w_{2}, \ldots, w_{n}]$$ as graded rings where $w_{2}, \ldots, w_{n}$ are the Stiefel-Whitney classes \cite[Theorem 11.1]{Tot}. In general, the rings $H^{*}_{\textup{dR}}(BG_{\mathbb{F}_{p}}/\mathbb{F}_{p})$ and $H^{*}(BG_{\mathbb{C}}, \mathbb{F}_{p})$ are different though. For example, $$\textup{dim}_{\mathbb{F}_{2}} H^{32}_{\textup{dR}}(B\textup{Spin}(11)_{\mathbb{F}_{2}}/\mathbb{F}_{2}) > \textup{dim}_{\mathbb{F}_{2}} H^{32}(B\textup{Spin}(11)_{\mathbb{C}}, \mathbb{F}_{2})$$ \cite[Theorem 12.1]{Tot}.

\indent  In this paper, we verify inequality \ref{comparison special fiber and generic} for more examples. For the torsion prime $2$ of the split reductive group $G_{2}$ over $\mathbb{Z}$, we show that $$H^{*}_{\textup{dR}}(B(G_{2})_{\mathbb{F}_{2}}/\mathbb{F}_{2}) \cong H^{*}(B(G_{2})_{\mathbb{C}}, \mathbb{F}_{2}) \cong \mathbb{F}_{2}[y_{4}, y_{6}, y_{7}]$$ as graded rings where $|y_{i}|=i$ for $i=4, 6, 7.$ For the spin groups, we show that \begin{equation} \label{isospin} H^{*}_{\textup{dR}}(B\textup{Spin}(n)_{\mathbb{F}_{2}}/\mathbb{F}_{2}) \cong H^{*}(B\textup{Spin}(n)_{\mathbb{C}}, \mathbb{F}_{2}) \end{equation} for $7 \leq n \leq 10$. Note that $2$ is a torsion prime for $\textup{Spin}(n)$ for $n \geq 7$. The isomorphism \ref{isospin} holds for $1 \leq n \leq 6$ by the ``accidental'' isomorphisms for spin groups along with \ref{nontorsion}.

\indent For $n=11$, we make a full computation of the de Rham cohomology of $B\textup{Spin}(n)_{\mathbb{F}_{2}}$: $$H_{\textup{dR}}^{*}(B\textup{Spin}(11)_{\mathbb{F}_{2}}/\mathbb{F}_{2}) \cong \mathbb{F}_{2}[y_4,y_6,y_7,y_8,y_{10},y_{11},y_{32}]/ (y_{7}y_{10}+y_{6}y_{11})$$ where $|y_{i}|=i$ for all $i$. We can compare this result with the computation of the singular cohomology of $B\textup{Spin}(11)_{\mathbb{C}}$ given by Quillen \cite{Quillen}:$$H^{*}(B\textup{Spin}(11)_{\mathbb{C}}, \mathbb{F}_{2})\cong \mathbb{F}_{2}[w_4,w_6,w_7,w_8,w_{10},w_{11},w_{64}]/ (w_{7}w_{10}+w_{6}w_{11},$$
\begin{flushright}$ w_{11}^{3}+w_{11}^{2}w_{7}w_{4}+w_{11}w_{8}w_{7}^{2})$\end{flushright}
where $|w_{i}|=i$ for all $i$. Equivalently, $$H^{*}(B\textup{Spin}(11)_{\mathbb{C}}, \mathbb{F}_{2})\cong H^{*}(BSO(11)_{\mathbb{C}}, \mathbb{F}_{2})/J \otimes \mathbb{F}_{2}[w_{64}]$$ where $J$ is the ideal generated by the regular sequence $$w_{2}, Sq^{1}(w_{2}), Sq^{2}Sq^{1}(w_{2}), \ldots, Sq^{16}Sq^{8} \cdots Sq^{1}w_{2}.$$ Thus, the rings $H^{*}_{\textup{dR}}(B\textup{Spin}(n)_{\mathbb{F}_{2}}/\mathbb{F}_{2})$ and $H^{*}(B\textup{Spin}(n)_{\mathbb{C}}, \mathbb{F}_{2})$ are not isomorphic in general even though $H^{*}_{\textup{dR}}(BSO(n)_{\mathbb{F}_{2}}/\mathbb{F}_{2}) \cong H^{*}(BSO(n)_{\mathbb{C}}, \mathbb{F}_{2})$ for all $n$. Steenrod squares on de Rham cohomology over a base field of characteristic $2$ have not yet been constructed. If they exist, our calculation suggests that their action on $H^{*}_{\textup{dR}}(BSO(n)_{\mathbb{F}_{2}}/\mathbb{F}_{2}) \cong H^{*}(BSO(n)_{\mathbb{C}}, \mathbb{F}_{2})$ would have to be different from the action of the topological Steenrod operations.

\section*{Acknowledgments} I thank Burt Totaro for suggesting this project to me and for providing advice. 
\section{Preliminaries}
\indent In this section, we recall results from \cite{Tot} that will be used in our computations. These results were also used by Totaro in \cite[Theorem 11.1]{Tot} to compute the de Rham cohomology of $BSO(n)_{k}$ for $k$ a field of characteristic $2$. 

\indent 

\indent The first result we mention \cite[Proposition 9.3]{Tot} is an analogue of the Leray-Serre spectral sequence from topology.

\begin{proposition} \label{Serre SS} Let $G$ be a split reductive group defined over a field $F$ and let $P$ be a parabolic subgroup of $G$ with Levi quotient $L$. There exists a spectral sequence of algebras $$E^{i,j}_{2}=H^{i}_{\textup{H}}(BG/F) \otimes H^{j}_{\textup{H}}((G/P)/F) \Rightarrow H^{i+j}_{\textup{H}}(BL/F).$$
\end{proposition}

Proposition \ref{Serre SS} is the main tool that we will use to compute Hodge cohomology rings of classifying stacks. To apply Proposition \ref{Serre SS}, we will choose a parabolic subgroup $P$ for which $H^{*}_{\textup{H}}(BL/F)$ is a polynomial ring.

\indent 
To fill in the $0$th column of the $E_{2}$ page in Proposition \ref{Serre SS}, we use a result of Srinivas \cite{Sri}.

\begin{proposition} \label{flagvarietyChow} Let $G$ be split reductive over a field $F$ and let $P$ be a parabolic subgroup of $G$. The cycle class map $$CH^{*}(G/P) \otimes _{\mathbb{Z}} F \to H^{*}_{\textup{H}}((G/P)/F)$$ is an isomorphism.
\end{proposition}
Under the cycle class map, $CH^{i}(G/P) \otimes _{\mathbb{Z}}F$ maps to $H^{i}(G/P, \Omega^{i})$. From the work of Chevalley \cite{Chev} and Demazure \cite{Dem}, $CH^{*}(G/P)$ is independent of the field $F$ and is isomorphic to the singular cohomology ring $H^{*}(G_{\mathbb{C}}/P_{\mathbb{C}}, \mathbb{Z})$. 

\indent The last piece of information we will use to compute $H^{*}_{\textup{H}}(BG/F)$ is the ring of $G$-invariants $O(\mathfrak{g})^{G}=\oplus_{i}H^{i}(BG, \Omega^{i})$. Let $T$  be a maximal torus in $G$ with Lie algebra $\mathfrak{t}$ and Weyl group $W$. There is a restriction homomorphism \begin{equation} \label{restriction} O(\mathfrak{g})^{G} \to O(\mathfrak{t})^{W}. \end{equation} We will need the following theorem which is due to Chaput and Romagny \cite[Theorem 1.1]{ChaRom}.

\begin{theorem} \label{theoreminvariants} Assume that $G$ is simple over a field $F$. Then the restriction homomorphism \ref{restriction} is an isomorphism unless $\textup{char}(F)=2$ and $G_{\overline{F}}$ is a product of copies of $Sp(2n)$ for some $n \in \mathbb{N}.$
\end{theorem}

\indent From the rings $O(\mathfrak{g})^{G}, CH^{*}(G/P), H^{*}_{\textup{H}}(BL/F)$, we will be able to determine the $E_{\infty}$ terms of the spectral sequence in Proposition \ref{Serre SS}. This will allow us to determine $H^{*}_{\textup{H}}(BG/F)$ by using the following version of the Zeeman comparison theorem \cite[Theorem VII.2.4]{MimTod}. 

\begin{theorem} \label{Zeeman}
Fix a field $F$. Let $\{\bar{E}_{r}^{i,j}\}, \{E_{r}^{i,j}\}$ be first quadrant (cohomological) spectral sequences of $F$-vector spaces such that $\bar{E}_{2}^{i,j}=\bar{E}_{2}^{i,0} \otimes _{F} \bar{E}_{2}^{0,j}$ and $E_{2}^{i,j}=E_{2}^{i,0} \otimes _{F} E_{2}^{0,j}$ for all $i,j$. Let $\{f_{r}^{i,j}: \bar{E}_{r}^{i,j} \to E_{r}^{i,j} \}$ be a morphism of spectral sequences such that $f^{i,j}_{2}=f^{i,0}_{2} \otimes f^{0,j}_{2}$ for all $i,j$. Fix $N, Q \in \mathbb{N}.$ Assume that $f_{\infty}^{i,j}$ is an isomorphism for all $i,j$ with $i+j<N$ and an injection for $i+j=N.$ If $f_{2}^{0,i}$ is an isomorphism for all $i<Q$ and an injection for $i=Q,$ then $f_{2}^{i,0}$ is an isomorphism for all $i<\min{(N, Q+1)}$ and an injection for $i=\min{(N,Q+1)}.$

\end{theorem}

\indent We recall a result from \cite[Section 11]{Tot} on the degeneration of the Hodge spectral sequence for split reductive groups, under some assumptions. The result in \cite[Section 11]{Tot} was proved for the special orthogonal groups but the proof works more generally.

\begin{proposition} \label{propHodgeSSG}
Let $G$ be a split reductive group over a field $F$ and assume that the Hodge cohomology ring of $BG$ is generated as an $F$-algebra by classes in $\oplus_{i} H^{i+1}(BG, \Omega^{i})$ and $\oplus_{i} H^{i}(BG, \Omega^{i}).$ Then the Hodge spectral sequence \begin{equation} \label{HodgeSSG} E^{i,j}_{1}=H^{j}(BG, \Omega^{i}) \Rightarrow H^{i+j}_{\textup{dR}}(BG/F)\end{equation} for $BG$ degenerates at the $E_{1}$ page.
\end{proposition}

\begin{proof}
From \cite[Lemma 8.2]{Tot}, there are natural maps $H^{i}(BG, \Omega^{i}) \to H^{2i}_{\textup{dR}}(BG/F)$ and $H^{i+1}(BG, \Omega^{i}) \to H^{2i+1}_{\textup{dR}}(BG/F)$ for all $i \geq 0$. These maps are compatible with products. Let $T$ denote a maximal torus of $G.$ From the group homomorphism $T \to G$, we have the commuting square

\begin{equation} \label{hodgessdegeneratesG_2diagram1}
\begin{tikzcd}
\oplus_{i} H^{i}(BG, \Omega^{i}) \arrow[r, ""] \arrow[d, ""] & \oplus_{i} H^{2i}_{\textup{dR}}(BG/F)\arrow[d, ""]\\
\oplus_{i} H^{i}(BT, \Omega^{i}) \arrow[r, "\cong"]& H^{2i}_{\textup{dR}}(BT/F).
\end{tikzcd}
\end{equation}
The restriction homomorphism \ref{restriction} induces an injection $$\oplus_{i} H^{i}(BG, \Omega^{i}) \to \oplus_{i} H^{i}(BT, \Omega^{i})$$ \cite[Lemma 8.2]{Tot}. Hence, from diagram \ref{hodgessdegeneratesG_2diagram1}, we get that the natural map $$\oplus_{i} H^{i}(BG, \Omega^{i}) \to \oplus_{i} H^{2i}_{\textup{dR}}(BG/F)$$ is an injection. Hence, any differentials into the diagonal in the spectral sequence \ref{HodgeSSG} must be $0$. Then all differentials on $\oplus_{i} H^{i+1}(BT, \Omega^{i})$ in \ref{HodgeSSG} must be $0$. All differentials on $\oplus_{i} H^{i}(BT, \Omega^{i})$ in the spectral sequence \ref{HodgeSSG} must be $0$ since $H^{i}(BG, \Omega^{j})=0$ for $i<j$ by \cite[Corollary 2.2]{Tot}. This proves that the Hodge spectral sequence for $BG$ degenerates.
\end{proof}

\section{$G_2$}

\indent Let $k$ be a field of characteristic $2$ and let $G$ denote the split form of $G_{2}$ over $k$.
\begin{theorem}  \label{cohofBG2}
 The Hodge cohomology ring of $BG$ is freely generated as a commutative $k$-algebra by generators $y_{4}\in H^{2}(BG,\Omega^{2})$, $y_{6}\in H^{3}(BG,\Omega^{3})$, and $y_{7}\in H^{4}(BG,\Omega^{3})$. The Hodge spectral sequence for $BG$ degenerates at $E_{1}$ and we have $$H^{*}_{\textup{dR}}(BG/k)\cong H_{\textup{H}}^{*}(BG/k)\cong k[y_{4}, y_{6}, y_{7}]$$.
\end{theorem}

\indent From the computation \cite[Corollary VII.6.3]{MimTod} of the singular cohomology ring of $B(G_{2})_{\mathbb{C}}$ with $\mathbb{F}_{2}$-coefficients, we then have $H^{*}(B(G_{2})_{\mathbb{C}}, k) \cong H^{*}_{\textrm{dR}}(BG/k)$. 

\begin{proof}
\indent We first choose a suitable parabolic subgroup of $G$. Let $P$ be the parabolic subgroup of $G$ corresponding to inclusion of the long root.
\begin{center}
\begin{tikzpicture}
\draw (2.5,0) circle (5pt) node[anchor=east]{};
\filldraw[black] (0,0) circle (5pt) node[anchor=west] {};
\draw[thick] (.15,.1) -- (2.35,.1);
\draw[thick] (.1,0) -- (2.33,0);
\draw[thick] (.15,-.1) -- (2.35,-.1);
\draw[thick] (1.245, 0) -- +(.2,.2);
\draw[thick] (1.245, 0) -- +(.2,-.2);
\end{tikzpicture}
\end{center}

\indent From Proposition \ref{flagvarietyChow}, $CH^{*}(G/P)$ is independent of the field $k$ and the characteristic of $k$. As discussed in \cite[\S 23.3]{FulHar}, if we consider $(G_{2})_{\mathbb{C}}$ over $\mathbb{C}$ along with the corresponding parabolic subgroup $P_{\mathbb{C}}$, $(G_{2})_{\mathbb{C}}/P_{\mathbb{C}}$ is isomorphic to a smooth quadric $Q_{5}$ in $\mathbb{P}^{6}$. Hence, by \cite[Chapter XIII]{EKM}, $H_{\textup{H}}^{*}((G/P)/k)$ is isomorphic to $$CH^{*}(Q_{5})\otimes_{\mathbb{Z}}{k}\cong k[v,w]/(v^6,w^2,v^3-2w) \cong k[v,w]/(v^3,w^2)$$ where $|v|=2$ and $|w|=6$ in $H_{\textup{H}}^{*}((G/P)/k)$ .

\indent We next show that the Levi quotient $L=P/R_{u}(P)$ of $P$ is isomorphic to $GL(2)_{k}$. This can be seen by constructing an isomorphism from the root datum of $GL(2)_{k}$ to the root datum of the Levi quotient. Let $(X_{1},R_{1},X_{1}^{\vee},R_{1}^{\vee})$ be the usual root datum of $GL(2)_{k}$ where $X_{1}=\mathbb{Z}\chi_{1}+\mathbb{Z}\chi_{2}$, $R_{1}=\mathbb{Z}(\chi_{1}-\chi_{2}),$ and we take our torus to be the set of diagonal matrices in $GL(2)_{k}$. We take $(X_{2},R_{2},X_{2}^{\vee},R_{2}^{\vee})$ to be the root datum of $G$ as described in \cite[Plate IX]{Bou}. Here, $X_{2}=\{(a,b,c)\in \mathbb{Z}^{3}\mid a+b+c=0\}$. The long root $\alpha$ for $G$ is then $(-2,1,1)$ and the root datum of $P/R_{u}(P)$ is $(X_{2},{\pm \alpha},X_{2}^{\vee},{\pm \frac{1}{3}\alpha})$. An isomorphism from the root datum of $GL(2)_{k}$ to the root datum of $G$ can then be obtained from the isomorphism $$X_{1} \to X_{2}$$ 
 $$\chi_{1}\longmapsto (-1,1,0), \chi_{2}\longmapsto (1,0,-1).$$ Thus, $L \cong GL(2)_{k}$.

\indent We now analyze the spectral sequence \begin{equation} \label{SerreSSG_2} E_{2}^{i,j}=H_{\textrm{H}}^{i}(BG/k)\otimes H_{\textrm{H}}^{j}((G/P)/k)\Rightarrow H_{\textrm{H}}^{i+j}(BL/k) \end{equation} from Proposition \ref{Serre SS}. From \cite[Theorem 4.1]{Tot}, $$H_{\textrm{H}}^{*}(BL/k)\cong S^{*}(\frak{gl_{2}})^{GL(2)_{k}}\cong S^{*}(\frak{t})^{S_{2}}\cong k[x_{1},x_{2}]$$ where $x_{1} \in H^{1}(BL,\Omega^{1})$ and $x_{2}\in H^{2}(BL,\Omega^{2}).$ Here, $\frak{t}$ is the space of all diagonal matrices in $\frak{gl_{2}}$ and $S_{2}$ acts on $\frak{t}$ by permuting the diagonal entries.

\indent In order to compute $H_{\textrm{H}}^{*}(BG/k)$ from the spectral sequence above, we must first compute the ring of invariants of $S^{*}(\frak{g}_{2})^{G}$. From Theorem \ref{theoreminvariants}, $S^{*}(\frak{g}_{2})^{G}\cong S^{*}(\frak{t_{0}})^{W}$ where $\frak{t_{0}}$ is the Lie algebra of a maximal torus $T$ in $G$ and $W$ is the corresponding Weyl group of $G$. By \cite[Corollary 2.2]{Tot}, $$H^{i}(BG,\Omega^{i})\cong S^{i}(\frak{t_{0}})^{W}$$ for $i\geq0$.
\begin{proposition} \label{invariantsofG_2} The ring of invariants $S^{*}(\frak{t_{0}})^{W}$ is isomorphic to $k[y_{4}, y_{6}]$ where $|y_{4}|=2$ and $|y_{6}|=3$ in $S^{*}(\frak{t_{0}})^{W}$ .
\end{proposition}
\begin{proof} Following the notation in \cite[Plate IX]{Bou}, $W\cong Z_{2}\times S_{3}$ acts on the root lattice $X_{2}=\{(a,b,c)\in \mathbb{Z}^{3}\mid a+b+c=0\}$ by multiplication by $-1$ and by permuting the coordinates. Hence, since we are working in characteristic $2$, $W$ acts on $S^{*}(\frak{t_{0}})=k[t_{1},t_{2},t_{3}]/(t_{1}+t_{2}+t_{3})$ by permuting $t_{1}$, $t_{2}$, and $t_{3}$. We then have 
$$S^{*}(\frak{t_{0}})^{W}\cong k[t_{1}t_{2}+t_{1}t_{3}+t_{2}t_{3},t_{1}t_{2}t_{3}]=k[y_{4},y_{6}].$$
\end{proof}
\indent We can now carry out the computation of $H_{\textrm{H}}^{*}(BG/k)$. First, we show that the class $v \in E^{0,2}_{2}$ is transgressive with $d_{3}(v)=0 \in E^{3,0}_{3}$. Consider the filtration on $H_{\textrm{H}}^{2}(BL/k)=k \cdot v$ given by \ref{SerreSSG_2}:

$$ H_{\textrm{H}}^{2}(BL/k) \overset{E^{0,2}_{\infty}} \hookleftarrow E^{2,0}_{\infty}.$$ Here, $E^{1,1}_{2}=0$ and $E^{2,0}_{\infty}=E^{2,0}_{2}=H_{\textrm{H}}^{2}(BG/k)=H^{1}(BG,\Omega^{1})$ since $H_{\textrm{H}}^{*}((G/P)/k)=\oplus _{i} H^{i}(G/P, \Omega^{i})$ is concentrated in even degrees. Hence, $$E^{2,0}_{\infty}=H_{\textrm{H}}^{2}(BG/k)=H^{1}(BG, \Omega^{1})=0$$ from Proposition \ref{invariantsofG_2}. It follows that $E^{0,2}_{\infty} \cong E^{0,2}_{2}=k \cdot v$ which implies that $d_{3}(v)=0$. As \ref{SerreSSG_2} is a spectral sequence of algebras, it follows that $d_{i}(v^2)=0$ for all $i \geq 2$. Using that $H_{\textrm{H}}^{*}(BL/k)$ is concentrated in even degrees, we then get that $H_{\textrm{H}}^{3}(BG/k)=E^{3,0}_{2}=E^{3,0}_{\infty}=0$ and $H_{\textrm{H}}^{5}(BG/k)=E^{5,0}_{2}=E^{5,0}_{\infty}=0.$

\indent Next, we show that $w \in H_{\textrm{H}}^{6}((G/P)/k)=E^{0,6}_{2}$ is transgressive with $0 \neq d_{7}(w) \in E^{7,0}_{7}. $ Note that $\textup{dim}_{k} H_{\textrm{H}}^{6}(BL/k)=2.$ As $v$ is transgressive with $d_{3}(v)=0$, we observe that $E^{4,2}_{\infty} \cong E^{4,2}_{2} \cong k \cdot y_{4} \otimes_{k} k \cdot v \cong k$ and $E^{6,0}_{\infty} \cong E^{6,0}_{2} \cong k \cdot y_{6} \cong k$. Hence, $\textup{dim}_{k} H_{\textrm{H}}^{6}(BL/k)=2=\textup{dim}_{k} E^{4,2}_{\infty} + \textup{dim}_{k} E^{6,0}_{\infty}.$ From the filtration on $H_{\textrm{H}}^{6}(BL/k)$ given by the spectral sequence \ref{SerreSSG_2}, it follows that $E^{0,6}_{\infty}=0.$ As $H_{\textrm{H}}^{3}(BG/k)=E^{3,0}_{2}=E^{3,0}_{\infty}=0$ and $H_{\textrm{H}}^{5}(BG/k)=E^{5,0}_{2}=E^{5,0}_{\infty}=0,$ we then get that $0 \neq d_{7}(w) \in E^{7,0}_{7}$ and $d_{7}(w)$ lifts to a non-zero element $y_{7} \in H^{4}(BG, \Omega^{3}) \subseteq H_{\textrm{H}}^{7}(BG/k).$

\[
\begin{tikzcd}
k \cdot w \arrow[ddddddrrrrrrr, ""] \\
0 & 0 & 0& 0 &0 &0& 0& 0\\
k \cdot v^{2} \arrow[ddddrrrrr, ""]   & 0& 0& 0& k \cdot v^{2}y_{4} & 0& k \cdot v^{2}y_{6} & k \cdot v^{2}y_{7}\\
0&0&0&0&0&0&0&0 \\
k \cdot v  \arrow[ddrrr,  ""] & 0& 0& 0& k \cdot vy_{4} & 0& k \cdot vy_{6} & k \cdot vy_{7}\\
0 & 0 & 0& 0 &0 &0& 0& 0\\
k & 0 & 0  & 0 & k \cdot y_{4}& 0& k \cdot y_{6}& k \cdot y_{7}
\end{tikzcd}
\]

\indent Now, we can determine the $E_{\infty}$ terms in \ref{SerreSSG_2}. For $n$ odd, $E^{i,n-i}_{\infty}=0$ since $H_{\textrm{H}}^{*}(BL/k)$ is concentrated in even degrees. Let $n \in \mathbb{N}$ be even. The $k$-dimension of $H_{\textrm{H}}^{n}(BL/k)$ is equal to the cardinality of the set $$S_{n}=\{(a,b) \in \mathbb{Z}_{\geq 0} \times \mathbb{Z}_{\geq 0}: 2a+4b=n\}.$$ For $i=0,1,2$, set $V_{i,n} \coloneqq H^{(n-2i)/2}(BG, \Omega^{(n-2i)/2}).$ For $i=0,1,2$, $\textup{dim}_{k} V_{i,n}$ is equal to the cardinality of the set $S_{i,n}=\{(a,b) \in \mathbb{Z}_{\geq 0} \times \mathbb{Z}_{\geq 0}: 4a+6b=n-2i\}.$ As $v$ is transgressive with $d_{3}(v)=0$, $E^{n-2i,2i}_{2} \cong E^{n-2i,2i}_{7}$ for $i=0,1,2.$ As $y_{7} \in H^{4}(BG, \Omega^{3})$ and $H^{i}(BG, \Omega^{j})=0$ for $i<j$,  $$y_{7}\cdot x \notin \oplus_{j} H^{j}(BG, \Omega^{j})$$ for all $x \in H_{\textrm{H}}^{*}(BG/k)$. Hence, $$H^{(n-2i)/2}(BG, \Omega^{(n-2i)/2}) \otimes_{k} k  \cdot v^{i} \subseteq E^{n-2i,2i}_{2} \cong E^{n-2i,2i}_{7}$$ injects into $E^{n-2i,2i}_{\infty}$ for $i=0,1,2.$

\indent Define a bijection $f_{n}:S_{n} \to S_{0,n} \cup S_{1,n} \cup S_{2,n}$ by 
\[f_{n}(a,b)= \begin{cases} 
      (b, a/3)\in S_{0,n} \, \, \textup{if} \, \, a \equiv 0 \mod{3}, \\          
      
      (b, (a-1)/3)\in S_{1,n} \, \, \textup{if} \, \, a \equiv 1 \mod{3}, \\          
      
      (b, (a-2)/3)\in S_{2.n} \, \, \textup{if} \, \, a \equiv 2 \mod{3}.  \\ 
   \end{cases}
\] Then $$\textup{dim}_{k} H_{\textrm{H}}^{n}(BL/k)=|S_{n}|=|S_{0,n}|+|S_{1,n}|+|S_{2,n}|$$ \\
$$=\textup{dim}_{k} H^{n/2}(BG, \Omega^{n/2}) + \textup{dim}_{k} H^{(n-2)/2}(BG, \Omega^{(n-2)/2}) + \textup{dim}_{k} H^{(n-4)/2}(BG, \Omega^{(n-4)/2})$$
\\ $$\leq \textup{dim}_{k} E^{n,0}_{\infty} + \textup{dim}_{k}E^{n-2,2}_{\infty} + \textup{dim}_{k}E^{n-4,4}_{\infty}$$ where the inequality follows from the fact proved above that $H^{(n-2i)/2}(BG, \Omega^{(n-2i)/2})$ injects into $E^{n-2i,2i}_{\infty}$ for $i=0,1,2.$ From the filtration on $H_{\textrm{H}}^{n}(BL/k)$ defined by the spectral sequence \ref{SerreSSG_2}, it follows that $H^{(n-2i)/2}(BG, \Omega^{(n-2i)/2}) \cong E^{n-2i,2i}_{\infty}$ for $i=0,1,2$ and $E^{n-2i,2i}_{\infty}=0$ for $i \geq 3.$

\indent We can now finish the computation of the Hodge cohomology of $BG$ by using Zeeman's comparison theorem. Let $F_{*}$ denote the cohomological spectral sequence of $k$-vector spaces with $E_{2}$ page concentrated on the $0$th column given by 

\[
F^{0,i}_{2}= \begin{cases}
k \, \, \textup{if} \, \, i=0, \\

k\cdot v \, \, \textup{if} \, \, i=2, \\

k\cdot v^{2} \, \, \textup{if} \, \, i=4, \\

0 \, \, \textup{if} \, \, i \neq 0,\,  2, \,  4. \\ \end{cases}
\] As $v \in E^{0,2}_{2}$ in the spectral sequence \ref{SerreSSG_2} is transgressive with $d_{r}(v)=0$ for all $r \geq 2$, there exists a map of of spectral sequences $F_{*} \to E_{*}$ that takes $v \in F^{0,2}_{2}$ to $v \in E^{0,2}_{2}$ and $v^{2} \in F^{0,4}_{2}$ to $v^{2} \in E^{0,4}_{2}$.

\begin{definition} For variables $x_{1}, \ldots, x_{n}$ let $\Delta(x_{1}, \ldots, x_{n})$ denote the $k$-vector space with basis given by the products $x_{i_{1}} \cdots x_{i_{r}}$ for $1 \leq i_{1} < i_{2}< \cdots < i_{r} \leq n.$
\end{definition}

\indent Fixing a variable $y$, let $H_{*}$ denote the cohomological spectral sequence with $E_{2}$ page given by $H_{2}=\Delta(w) \otimes k[y]$ where $w$ is of bidegree $(0,6)$, $y$ is of bidegree $(7,0)$, and $w$ is transgressive with $d_{7}(wy^{i})=y^{i+1}$ for all $i \geq 0$. As $w \in E^{0,6}_{2}$ is transgressive with $d_{7}(w)=y_{7} \in E^{7,0}_{2}$, there exists a map of spectral sequence $H_{*} \to E_{*}$ such that $w \in H^{0,6}_{2}$ maps to $w \in E^{0,6}_{2}$ and $y \in H^{7,0}_{2}$ maps to $y_{7} \in E^{7,0}_{2}.$ Elements of the ring of $G$-invariants $k[y_{4}, y_{6}]$ are permanent cycles in the spectral sequence \ref{SerreSSG_2} since they are concentrated on the $0$th row. Thus, by tensoring the previous maps of spectral sequences, we get a map $$\alpha: I_{*} \coloneqq F_{*} \otimes H_{*} \otimes k[y_{4}, y_{6}] \to E_{*}$$ of spectral sequences.

\indent 
As shown above, the map $\alpha$ induces an isomorphism $I_{\infty} \cong F_{2} \otimes k[y_{4}, y_{6}] \to E_{\infty}$ on $E_{\infty}$ pages. The $0$th columns of the $E_{2}$ pages of the spectral sequences $I_{*}$ and $E_{*}$ are both isomorphic to $k[v,w]/(v^{3}, w^{2})$ and $\alpha$ induces an isomorphism on the $0$th columns of the $E_{2}$ pages. Thus, by the Zeeman comparison theorem \ref{Zeeman}, $\alpha$ induces an isomorphism on the $0$th rows of the $E_{2}$ pages. Hence, $$ H_{\textup{H}}^{*}(BG/k)\cong k[y_{4}, y_{6}, y_{7}].$$ From Proposition \ref{propHodgeSSG}, the Hodge spectral sequence for $BG$ degenerates.

\end{proof}
\begin{corollary} \label{corollaryformcoh}
Let $G$ be a $k$-form of $G_{2}$. Then $$H_{\textup{H}}^{*}(BG/k)\cong k[x_{4}, x_{6}, x_{7}]$$ where $|x_{i}|=i$ for $i=4,6,7.$
\end{corollary}

\begin{proof}

\indent Letting $k_{s}$ denote the separable closure of $k$, we have $BG \times _{k} \textup{Spec}(k_{s}) \cong BG_{2} \times _{k} \textup{Spec}(k_{s}).$ From Theorem \ref{cohofBG2}, $H_{\textup{H}}^{*}((BG_{2} \times _{k} \textup{Spec}(k_{s}))/k_{s})\cong k_{s}[x^{'}_{4}, x^{'}_{6}, x^{'}_{7}]$ for some $x^{'}_{4}, x^{'}_{6}, x^{'}_{7} \in H_{\textup{H}}^{*}(BG_{2} \times _{k} \textup{Spec}(k_{s})/k_{s})$ with $|x^{'}_{i}|=i$ for all $i.$ As Hodge cohomology commutes with extensions of the base field, $$H_{\textup{H}}^{*}((BG \times _{k} \textup{Spec}(k_{s}))/k_{s}) \cong  H_{\textup{H}}^{*}(BG/k) \otimes_{k} k_{s} .$$ It follows that $H_{\textup{H}}^{*}(BG/k)\cong k[x_{4}, x_{6}, x_{7}]$ for some $x_{4}, x_{6}, x_{7} \in H_{\textup{H}}^{*}(BG/k).$
\end{proof}

\section{Spin groups}

Let $k$ be a field of characteristic $2$ and let $G$ denote the split group $\textup{Spin}(n)$ over $k$ for $n \geq 7.$

\indent Let $P_{0} \subset SO(n)_{k}$ denote a parabolic subgroup that stabilizes a maximal isotropic subspace. Let $P \subset G$ denote the inverse image of $P_{0}$ under the double cover map $G \to SO(n)_{k}.$ The Hodge cohomology of $G/P$ is given by Proposition \ref{flagvarietyChow} and \cite[Theorem III.6.11]{MimTod}.

\begin{proposition} \label{cohflagvar}
There is an isomorphism

$$H^{*}_{\textup{H}}((G/P)/k) \cong k[e_{1}, \ldots, e_{s}]/(e_{i}^{2}=e_{2i}),$$

where $s=\lfloor (n-1)/2 \rfloor,$ $e_{m}=0$ for $m>s,$ and $|e_{i}|=2i$ for all $i.$
\end{proposition}

\indent The Levi quotient of $P_{0}$ is isomorphic to $GL(r)_{k}$ where $r=\lfloor n/2 \rfloor$. Hence, the Levi quotient $L$ of $P$ is a double cover of $GL(r)_{k}.$

\begin{proposition} \label{propmetalinear}
The torsion index of $L$ is equal to $1.$
\end{proposition}

\begin{proof}
We show that the torsion index of the corresponding compact connected Lie group $M$ is equal to $1.$ As $M$ is a double cover of $U(r)$, $M$ is isomorphic to $(S^{1} \times SU(r))/2\mathbb{Z}$ where $k \in \mathbb{Z}$ acts on $S^{1} \times SU(r)$ by $$(z,A)\mapsto (ze^{2\pi ik/r},e^{-2\pi ik/r}A).$$ Hence, the derived subgroup $ \left[M,M\right]$ of $M$ is isomorphic to $SU(r)$. As $SU(r)$ has torsion index $1,$ $M$ has torsion index $1$ by \cite[Lemma 2.1]{Tot1}. Thus, $L$ has torsion index equal to $1.$
\end{proof}

\begin{corollary} \label{cohofBL}
There is an isomorphism $$H^{*}_{\textup{H}}(BL/k) \cong O(\frak{l})^{L} \cong k[A, c_{2}, \ldots, c_{r}]$$ where $|c_{i}|=2i$ in $H^{*}_{\textup{H}}(BL/k)$ for all $i$ and $|A|=2.$
\end{corollary}
\begin{proof}

\indent 
From Proposition \ref{propmetalinear} and \cite[Theorem 9.1]{Tot}, $$H^{*}_{\textup{H}}(BL/k) \cong O(\frak{l})^{L}.$$ Let $T$ be a maximal torus in $L$ with Lie algebra $\frak{t}$ and Weyl group $W.$ From Theorem \ref{theoreminvariants}, $O(\frak{l})^{L} \cong O(\frak{t})^{W}.$ To compute $O(\frak{t})^{W},$ we use that $L$ is a double cover of $GL(r)_{k}.$  We have $$S(X^{*}(T) \otimes k) \cong \mathbb{Z}[x_{1}, \ldots, x_{r},A]/(2A=x_{1}+\cdots +x_{r}) \otimes k \cong k[x_{1}, \ldots, x_{r}, A]/(x_{1}+\cdots +x_{r}).$$ The Weyl group $W$ of $L$ is isomorphic to the symmetric group $S_{r}$ and acts on $S(X^{*}(T)\otimes k) $ by permuting $x_{1}, \ldots, x_{r}.$ From \cite[Proposition 4.1]{Nak}, $$(k[x_{1}, \ldots, x_{r}, A]/(x_{1}+\cdots +x_{r}))^{S_{r}}=k[A, c_{2}, \ldots, c_{r}]$$ where $c_{1}, \ldots, c_{r}$ are the elementary symmetric polynomials in the variables $x_{1}, \ldots, x_{r}.$
\end{proof}

\begin{theorem} \label{spin7prop} Let $n=7.$ The Hodge spectral sequence for $BG$ degenerates and $$H^{*}_{\textup{dR}}(BG/k)\cong H_{\textup{H}}^{*}(BG/k)\cong k[y_{4}, y_{6}, y_{7}, y_{8}]$$ where $|y_{i}|=i$ for $i=4,6,7,8.$
\end{theorem}

\begin{proof} 
\indent From \cite[Section 12]{Tot}, $$O(\frak{g})^{G} \cong k[y_{4}, y_{6}, y_{8}]$$ where $|y_{i}|=i$ in $H_{\textup{H}}^{*}(BG/k)$, viewing  $O(\frak{g})^{G}$ as a subring of $H_{\textup{H}}^{*}(BG/k).$ Consider the spectral sequence \begin{equation} \label{SSSpin7} E_{2}^{i,j}=H_{\textup{H}}^{i}(BG/k)\otimes H_{\textup{H}}^{j}((G/P)/k)\Rightarrow H_{\textup{H}}^{i+j}(BL/k) \end{equation} from Proposition \ref{Serre SS}. From Proposition \ref{cohflagvar} and Corollary \ref{cohofBL}, $$H_{\textup{H}}^{*}((G/P)/k) \cong k[e_{1},e_{2},e_{3}]/(e_{i}^{2}=e_{2i})$$ and $$H_{\textup{H}}^{*}(BL/k) \cong k[A, c_{2}, c_{3}].$$   

\indent First, we show that $e_{1} \in E^{0,2}_{2}$ is transgressive with $d_{3}(e_{1})=0.$ From the filtration on $H_{\textup{H}}^{2}(BL/k)=k \cdot A$ given by \ref{SSSpin7}, we have $$1=\textup{dim}_{k} E^{0,2}_{\infty}+\textup{dim}_{k} E^{2,0}_{\infty}=\textup{dim}_{k} E^{0,2}_{\infty}+\textup{dim}_{k} E^{2,0}_{2}.$$ As $H_{\textup{H}}^{*}(BL/k)= \oplus_{i}H^{i}(BL,\Omega^{i}),$ $E^{2,0}_{2}=H^{1}(BG,\Omega^{1})=0.$ Hence, $E^{0,2}_{\infty}=E^{0,2}_{2}=k \cdot e_{1}$ which implies that all differentials on $e_{1}$ are $0.$ As $e_{2}=e_{1}^{2},$ it follows that all differentials in \ref{SSSpin7} are $0$ on $e_{2} \in E^{0,4}_{2}.$ Hence, $E^{4,2}_{\infty} \cong E^{4,2}_{2} \cong k \cdot (y_{4} \otimes e_{1})$ and $E^{6,0}_{\infty} \cong E^{6,0}_{2} \cong k \cdot y_{6}.$

\indent We next show that $e_{3} \in E^{0,6}_{2}$ is transgressive with $d_{7}(e_{3}) \neq 0.$ As all differentials on $e_{1}$ are $0$ and $H_{\textup{H}}^{i}(BL/k)=0$ for $i$ odd, the spectral sequence \ref{SSSpin7} implies that $E^{3,0}_{2}=E^{5,0}_{2}=0.$ Consider the filtration of \ref{SSSpin7} on $H_{\textup{H}}^{6}(BL/k).$ We have $$\textup{dim}_{k}H_{\textup{H}}^{6}(BL/k)=3=\textup{dim}_{k} E^{6,0}_{\infty} + \textup{dim}_{k} E^{4,2}_{\infty} + \textup{dim}_{k} E^{0,6}_{\infty}=2+\textup{dim}_{k} E^{0,6}_{\infty}$$ which implies that $E^{0,6}_{\infty} \cong  k \cdot e_{1}e_{2}.$ As $E^{3,0}_{2}=E^{5,0}_{2}=0,$ we must then have $e_{3} \in E^{0,6}_{7}$ and $0 \neq d_{7}(e_{3}) \in E^{7,0}_{7}.$ The class $d_{7}(e_{3})$ lifts to a non-zero class $y_{7} \in H^{4}(BG, \Omega^{3}) \subseteq E^{7,0}_{2}=H_{\textup{H}}^{7}(BG/k).$

\[
\begin{tikzcd}
 k\cdot e_{1}e_{2} \oplus k \cdot e_{3} \arrow[ddddddrrrrrrr, ""] \\
0 & 0 & 0& 0 &0 &0& 0& 0\\
k \cdot e_{2} \arrow[ddddrrrrr, ""]   & 0& 0& 0& k \cdot e_{2}y_{4} & 0& k \cdot e_{2}y_{6} & k \cdot e_{2}y_{7}\\
0&0&0&0&0&0&0&0 \\
k \cdot e_{1}  \arrow[ddrrr,  ""] & 0& 0& 0& k \cdot e_{1}y_{4} & 0& k \cdot e_{1}y_{6} & k \cdot e_{1}y_{7}\\
0 & 0 & 0& 0 &0 &0& 0& 0\\
k & 0 & 0  & 0 & k \cdot y_{4}& 0& k \cdot y_{6}& k \cdot y_{7}
\end{tikzcd}
\]

\indent We can now determine the $E_{\infty}$ page of \ref{SSSpin7}. For $n$ odd, $E^{i,n-i}_{\infty}=0$ since $H^{*}_{\textup{H}}(BL/k)$ is concentrated in even degrees. Assume that $n \in \mathbb{N}$ is even. The $k$-dimension of $H_{\textrm{H}}^{n}(BL/k)$ is equal to the cardinality of the set $$S_{n}=\{(a,b,c) \in \mathbb{Z}_{\geq 0} \times \mathbb{Z}_{\geq 0} \times \mathbb{Z}_{\geq 0}: 2a+4b+6c=n\}.$$ For $i=0,1,2,3$, set $V_{i,n} \coloneqq H^{(n-2i)/2}(BG, \Omega^{(n-2i)/2}).$ For $i=0,1,2,3$, $\textup{dim}_{k} V_{i,n}$ is equal to the cardinality of the set $S_{i,n}=\{(a,b,c) \in \mathbb{Z}_{\geq 0} \times \mathbb{Z}_{\geq 0} \times \mathbb{Z}_{\geq 0}: 4a+6b+8c=n-2i\}.$ As all differentials in \ref{SSSpin7} are $0$ on $e_{1} \in H^{2}_{\textup{H}}((G/P)/k),$ $$V_{i,n} \cong V_{i,n} \otimes k\cdot e_{1}^{i} \subseteq E^{n-2i,2i}_{\infty}$$ for $i=0,1,2,3$.

\indent 
Define a bijection $f_{n}:S_{n} \to S_{0,n} \cup S_{1,n} \cup S_{2,n} \cup S_{3,n}$ by 
\[f_{n}(a,b,c)= \begin{cases} 
      (b,c, a/4)\in S_{0,n} \, \, \textup{if} \, \, a \equiv 0 \mod{4}, \\          
      
      (b,c, (a-1)/4)\in S_{1,n} \, \, \textup{if} \, \, a \equiv 1 \mod{4}, \\          
      
      (b,c, (a-2)/4)\in S_{2,n} \, \, \textup{if} \, \, a \equiv 2 \mod{4},  \\ 
      
      (b,c, (a-3)/4)\in S_{3,n} \, \, \textup{if} \, \, a \equiv 3 \mod{4}.
   \end{cases}
\] Then $$\textup{dim}_{k} H_{\textrm{H}}^{n}(BL/k)=|S_{n}|=|S_{0,n}|+|S_{1,n}|+|S_{2,n}|++|S_{3,n}|.$$ As $$\textup{dim}_{k} H_{\textrm{H}}^{n}(BL/k) \geq E^{n,0}_{\infty}+ E^{n-2,2}_{\infty}+ E^{n-4,4}_{\infty}+ E^{n-6,6}_{\infty}$$ and $V_{i,n} \subseteq E^{n-2i,2i}_{\infty}$ for $i=0,1,2,3$, it follows that $V_{i,n} \cong E^{n-2i,2i}_{\infty}$ for $i=0,1,2,3$ and $E^{n-2i,2i}_{\infty}=0$ for $i \geq 4.$
 
\indent We now use Zeeman's comparison theorem \ref{Zeeman} to finish the computation of the Hodge cohomology of $BG.$ Let $F_{*}$ denote the cohomological spectral sequence of $k$-vector spaces with $E_{2}$ page concentrated on the $0$th column given by $F_{2}=\Delta(e_{1},e_{2})$ where $e_{i}$ is of bidegree $(0,2i)$ for $i=1,2.$ As all differentials are $0$ on $e_{1} \in E^{0,2}_{2}$ in the spectral sequence \ref{SSSpin7}, there is a map of spectral sequences $F_{*} \to E_{*}$ taking $e_{i} \in  F^{0,2i}_{2}$ to $e_{i} \in  E^{0,2i}_{2}$ for $i=1,2.$ Fix a variable $y.$ Let $H_{*}$ be the spectral sequence with $E_{2}$ page given by $H_{2}=\Delta(e_{3}) \otimes k[y]$ where $e_{3}$ is of bidegree $(0,6)$, $y$ is of bidegree $(7,0),$ and $e_{3}$ is transgressive with $d_{7}(e_{3}y^{i})=y^{i+1}$ for all $i.$ As $e_{3} \in E^{0,6}_{2}$ is transgressive with $d_{7}(e_{3})=y_{7},$ there exists a map of spectral sequences $H_{*} \to E_{*}$ taking $e_{3} \in H^{0,6}_{2}$ to $e_{3} \in E^{0,6}_{2}$ and $y \in H^{7,0}_{2}$ to $y_{7} \in E^{7,0}_{2}.$ 

\indent Elements in the ring of $G$-invariants $k[y_{4}, y_{6}, y_{8}]$ are permanent cycles in the spectral sequence \ref{SSSpin7}. Tensoring maps of spectral sequences, we get a map $$\alpha: I_{*} \coloneqq F_{*} \otimes H_{*} \otimes k[y_{4}, y_{6},y_{8}] \to E_{*}$$ of spectral sequences. As $I_{\infty} \cong F_{2} \otimes k[y_{4}, y_{6},y_{8}],$ $\alpha$ induces isomorphisms on $E_{\infty}$ terms and on the $0$th columns of the $E_{2}$ pages. Hence, by Zeeman's comparison theorem \ref{Zeeman}, $\alpha$ induces an isomorphism on the $0$th rows of the $E_{2}$ pages. Thus, $$H^{*}_{\textup{H}}(BG/k)\cong k[y_{4}, y_{6}, y_{7},y_{8}].$$ The Hodge spectral sequence for $BG$ degenerates by Proposition \ref{propHodgeSSG}.
\end{proof}

\indent As Hodge cohomology commutes with extensions of the base field, we have the following result.

\begin{corollary} Let $k$ be a field of characteristic $2$ and let $G$ be a $k$-form of $\textup{Spin}(7).$ Then $$H^{*}_{\textup{H}}(BG/k)\cong k[x_{4},x_{6},x_{7},x_{8}]$$ where $|x_{i}|=i$ for all $i.$
\end{corollary}

\begin{theorem}
Let $n=8.$ The Hodge spectral sequence for $BG$ degenerates and $$H^{*}_{\textup{dR}}(BG/k)\cong H_{\textup{H}}^{*}(BG/k)\cong k[y_{4}, y_{6}, y_{7}, y_{8}, y_{8}']$$ where $|y_{i}|=i$ for $i=4,6,7,8$ and $|y_{8}'|=8.$
\end{theorem}

\begin{proof}
\indent From \cite[Section 12]{Tot}, $$O(\frak{g})^{G} \cong k[y_{4}, y_{6}, y_{8},y_{8}']$$ where $|y_{i}|=i$ and $|y_{8}'|=8$ in $H_{\textup{H}}^{*}(BG/k)$, viewing  $O(\frak{g})^{G}$ as a subring of $H_{\textup{H}}^{*}(BG/k).$ Consider the spectral sequence \begin{equation} \label{SSSpin8} E_{2}^{i,j}=H_{\textup{H}}^{i}(BG/k)\otimes H_{\textup{H}}^{j}((G/P)/k)\Rightarrow H_{\textup{H}}^{i+j}(BL/k) \end{equation} from Proposition \ref{Serre SS}. From Proposition \ref{cohflagvar} and Corollary \ref{cohofBL}, $$H_{\textup{H}}^{*}((G/P)/k) \cong k[e_{1},e_{2},e_{3}]/(e_{i}^{2}=e_{2i})$$ and $$H_{\textup{H}}^{*}(BL/k) \cong k[A, c_{2}, c_{3},c_{4}].$$ Calculations similar to those performed in the proof of Proposition \ref{spin7prop} show that all differentials are zero on $e_{1}$ and $e_{3}\in E^{0,6}_{2}$ is transgressive with $0 \neq d_{7}(e_{3})=y_{7} \in H^{4}(BG, \Omega^{3}).$  We have $H^{m}_{\textup{H}}(BG/k) \cong H^{m}_{\textup{H}}(B\textup{Spin}(7)_{k}/k)$ for $m<8$ and $H^{8}_{\textup{H}}(BG/k)=k \cdot y_{8} \oplus k \cdot y_{8}'.$

\indent We can now determine the $E_{\infty}$ terms for \ref{SSSpin8}. For $n$ odd, $E^{i,n-i}_{\infty}=0$ since $H^{*}_{\textup{H}}(BL/k)$ is concentrated in even degrees. Assume that $n \in \mathbb{N}$ is even. The $k$-dimension of $H_{\textrm{H}}^{n}(BL/k)$ is equal to the cardinality of the set $$S_{n}=\{(a,b,c,d) \in \mathbb{Z}_{\geq 0} \times \mathbb{Z}_{\geq 0} \times \mathbb{Z}_{\geq 0}\times \mathbb{Z}_{\geq 0}: 2a+4b+6c+8d=n\}.$$ For $i=0,1,2,3$, set $V_{i,n} \coloneqq H^{(n-2i)/2}(BG, \Omega^{(n-2i)/2}).$ For $i=0,1,2,3$, $\textup{dim}_{k} V_{i,n}$ is equal to the cardinality of the set $S_{i,n}=\{(a,b,c,d) \in \mathbb{Z}_{\geq 0} \times \mathbb{Z}_{\geq 0} \times \mathbb{Z}_{\geq 0}\times \mathbb{Z}_{\geq 0}: 4a+6b+8c+8d=n-2i\}.$ As all differentials in \ref{SSSpin8} are $0$ on $e_{1} \in H^{2}_{\textup{H}}((G/P)/k),$ $$V_{i,n} \cong V_{i,n} \otimes k\cdot e_{1}^{i} \subseteq E^{n-2i,2i}_{\infty}$$ for $i=0,1,2,3$.

\indent 
Define a bijection $f_{n}:S_{n} \to S_{0,n} \cup S_{1,n} \cup S_{2,n} \cup S_{3,n}$ by 
\[f_{n}(a,b,c,d)= \begin{cases} 
      (b,c,d, a/4)\in S_{0,n} \, \, \textup{if} \, \, a \equiv 0 \mod{4}, \\          
      
      (b,c,d, (a-1)/4)\in S_{1,n} \, \, \textup{if} \, \, a \equiv 1 \mod{4}, \\          
      
      (b,c,d, (a-2)/4)\in S_{2,n} \, \, \textup{if} \, \, a \equiv 2 \mod{4},  \\ 
      
      (b,c,d, (a-3)/4)\in S_{3,n} \, \, \textup{if} \, \, a \equiv 3 \mod{4}.
   \end{cases}
\] Then $$\textup{dim}_{k} H_{\textrm{H}}^{n}(BL/k)=|S_{n}|=|S_{0,n}|+|S_{1,n}|+|S_{2,n}|++|S_{3,n}|.$$ As $$\textup{dim}_{k} H_{\textrm{H}}^{n}(BL/k) \geq E^{n,0}_{\infty}+ E^{n-2,2}_{\infty}+ E^{n-4,4}_{\infty}+ E^{n-6,6}_{\infty}$$ and $V_{i,n} \subseteq E^{n-2i,2i}_{\infty}$ for $i=0,1,2,3$, it follows that $V_{i,n} \cong E^{n-2i,2i}_{\infty}$ for $i=0,1,2,3$ and $E^{n-2i,2i}_{\infty}=0$ for $i \geq 4.$

\indent Let $F_{*}$ denote the spectral sequence with $F_{2}=\Delta(e_{1},e_{2},e_{4})$ where $e_{i}$ is of bidegree $(0,2i).$ There is a map of spectral sequences $F_{*} \to E_{*}$ taking $e_{i}$ to $e_{i}$ for $i=1,2,4.$ Fix a variable $y.$ Let $H_{*}$ denote the spectral sequence with $E_{2}$ page $H_{2}=\Delta(e_{3}) \otimes k[y]$ where $e_{3}$ is of bidegree $(0,6),$ $y$ is of bidegree $(7,0),$ and $e_{3}$ is transgressive with $d_{7}(e_{3}y^{i})=y^{i+1}$ for all $i.$ There is an obvious map of spectral sequences $H_{*} \to E_{*}.$ Classes in the ring of $G$-invariants are permanent cycles in the spectral sequence \ref{SSSpin8}. Tensoring these maps, we get a map of spectral sequences $$\alpha: I_{*} \coloneqq F_{*} \otimes H_{*} \otimes k[y_{4}, y_{6},y_{8},y_{8}'] \to E_{*}.$$

\indent The map $\alpha$ induces an isomorphism on $E_{\infty}$ terms and on the $0$th columns of the $E_{2}$ pages. Zeeman's comparison theorem \ref{Zeeman} then implies that $\alpha$ induces an isomorphism on the $0$th rows of the $E_{2}$ pages. Thus, $$H_{\textup{H}}^{*}(BG/k)\cong k[y_{4}, y_{6}, y_{7}, y_{8}, y_{8}'].$$ Proposition \ref{propHodgeSSG} implies that the Hodge spectral sequence for $BG$ degenerates.

\end{proof}

\begin{corollary}
Let $k$ be a field of characteristic $2$ and let $G$ be a $k$-form for $\textup{Spin}(8).$ Then $$H_{\textup{H}}^{*}(BG/k)\cong k[y_{4}, y_{6}, y_{7}, y_{8}, y_{8}']$$ where $|y_{i}|=i$ for $i=4,6,7,8$ and $|y_{8}'|=8.$
\end{corollary}

\begin{theorem} \label{theoremspin10}
Let $n=9.$ The Hodge spectral sequence for $BG$ degenerates and $$H^{*}_{\textup{dR}}(BG/k)\cong H_{\textup{H}}^{*}(BG/k)\cong k[y_{4}, y_{6}, y_{7}, y_{8}, y_{16}]$$ where $|y_{i}|=i$ for $i=4,6,7,8,16.$ 
\end{theorem}

\begin{proof}
\indent From \cite[Section 12]{Tot}, $$O(\frak{g})^{G} \cong k[y_{4}, y_{6}, y_{8},y_{16}]$$ where $|y_{i}|=i$ in $H_{\textup{H}}^{*}(BG/k)$, viewing  $O(\frak{g})^{G}$ as a subring of $H_{\textup{H}}^{*}(BG/k).$ Consider the spectral sequence \begin{equation} \label{SSSpin9} E_{2}^{i,j}=H_{\textup{H}}^{i}(BG/k)\otimes H_{\textup{H}}^{j}((G/P)/k)\Rightarrow H_{\textup{H}}^{i+j}(BL/k) \end{equation} from Proposition \ref{Serre SS}. From Proposition \ref{cohflagvar} and Corollary \ref{cohofBL}, $$H_{\textup{H}}^{*}((G/P)/k) \cong k[e_{1},e_{2},e_{3},e_{4}]/(e_{i}^{2}=e_{2i})$$ and $$H_{\textup{H}}^{*}(BL/k) \cong k[A, c_{2}, c_{3},c_{4}].$$ Calculations similar to those performed in the proof of Proposition \ref{spin7prop} show that all differentials in \ref{SSSpin9} are zero on $e_{1}$ and $e_{3}\in E^{0,6}_{2}$ is transgressive with $0 \neq d_{7}(e_{3})=y_{7} \in H^{4}(BG, \Omega^{3}).$ We have $H^{m}_{\textup{H}}(BG/k) \cong H^{m}_{\textup{H}}(B\textup{Spin}(7)_{k}/k)$ for $m \leq 10.$

\indent We now determine the $E_{\infty}$ terms for \ref{SSSpin9}. For $n$ odd, $E^{i,n-i}_{\infty}=0$ since $H^{*}_{\textup{H}}(BL/k)$ is concentrated in even degrees. Assume that $n \in \mathbb{N}$ is even. The $k$-dimension of $H_{\textrm{H}}^{n}(BL/k)$ is equal to the cardinality of the set $$S_{n}=\{(a,b,c,d) \in \mathbb{Z}_{\geq 0} \times \mathbb{Z}_{\geq 0} \times \mathbb{Z}_{\geq 0}\times \mathbb{Z}_{\geq 0}: 2a+4b+6c+8d=n\}.$$ For $0\leq i \leq 7$, set $V_{i,n} \coloneqq H^{(n-2i)/2}(BG, \Omega^{(n-2i)/2}).$ For $0\leq i \leq 7$, $\textup{dim}_{k} V_{i,n}$ is equal to the cardinality of the set $S_{i,n}=\{(a,b,c,d) \in \mathbb{Z}_{\geq 0} \times \mathbb{Z}_{\geq 0} \times \mathbb{Z}_{\geq 0}\times \mathbb{Z}_{\geq 0}: 4a+6b+8c+16d=n-2i\}.$ As all differentials in \ref{SSSpin9} are $0$ on $e_{1} \in H^{2}_{\textup{H}}((G/P)/k),$ $$V_{i,n} \cong V_{i,n} \otimes k\cdot e_{1}^{i} \subseteq E^{n-2i,2i}_{\infty}$$ for $0\leq i \leq 7$.

\indent 
Define a bijection $f_{n}:S_{n} \to \bigcup\limits_{i=0}^{7} S_{i,n}$ by 
$f_{n}(a,b,c,d)=(b,c,d,(a-i)/8)\in S_{i,n}$ for $a \equiv i \mod(8).$ Then $$\textup{dim}_{k} H_{\textrm{H}}^{n}(BL/k)=|S_{n}|=\sum_{i=0}^{7}|S_{i,n}|.$$ As $$\textup{dim}_{k} H_{\textrm{H}}^{n}(BL/k) \geq \sum_{i=0}^{7} E^{n-2i,2i}_{\infty}$$ and $V_{i,n} \subseteq E^{n-2i,2i}_{\infty}$ for $0\leq i \leq 7$, it follows that $V_{i,n} \cong E^{n-2i,2i}_{\infty}$ for $0\leq i \leq 7$ and $E^{n-2i,2i}_{\infty}=0$ for $i \geq 8.$

\indent Let $F_{*}$ denote the cohomological spectral sequence with $E_{2}$ page given by $F_{2}=\Delta(e_{1},e_{2},e_{4})$ where $e_{i}$ has bidegree $(0,2i)$ for $i=1,2,4.$ As all differentials in the spectral sequence \ref{SSSpin9} are $0$ on $e_{1},$ there exists a map $F_{*} \to E_{*}$ of spectral sequences taking $e_{i}$ to $e_{i}$ for $i=1,2,4.$ Let $y$ be a free variable and let $H_{*}$ denote the spectral sequence with $E_{2}$ page $H_{2}=\Delta(e_{3})\otimes k[y]$ where $e_{3}$ is of bidegree $(0,6),$ $y$ is of bidegree $(7,0),$ and $e_{3}$ is transgressive with $d_{7}(e_{3}y^{i})=y^{i+1}$ for all $i.$ As $e_{3}$ is transgressive in the spectral sequence \ref{SSSpin9} with $d_{7}(e_{3})=y_{7},$ there exists a map of spectral sequences $H_{*} \to E_{*}$ taking $e_{3}$ to $e_{3}$ and $y$ to $y_{7}.$ 

\indent Elements in the ring of $G$-invariants $k[y_{4},y_{6},y_{8},y_{16}]$ are permanent cycles in the spectral sequence \ref{SSSpin9}. Tensoring maps of spectral sequences, we get a map $$\alpha: I_{*} \coloneqq F_{*} \otimes H_{*} \otimes k[y_{4}, y_{6},y_{8},y_{16}] \to E_{*}.$$ The map $\alpha$ induces an isomorphism on $E_{\infty}$ terms and on the $0$th columns of the $E_{2}$ pages. Hence, Zeeman's comparison theorem \ref{Zeeman} implies that $\alpha$ induces an isomorphism on the $0$th rows of the $E_{2}$ pages. Thus, $$H_{\textup{H}}^{*}(BG/k)\cong k[y_{4}, y_{6}, y_{7}, y_{8}, y_{16}].$$ Proposition \ref{HodgeSSG} implies that the Hodge spectral sequence for $BG$ degenerates.

\end{proof}

\begin{corollary}
Let $k$ be a field of characteristic $2$ and let $G$ be a $k$-form for $\textup{Spin}(9).$ Then $$H_{\textup{H}}^{*}(BG/k)\cong k[y_{4}, y_{6}, y_{7}, y_{8}, y_{16}]$$ where $|y_{i}|=i$ for $i=4,6,7,8,16.$
\end{corollary}

\begin{remark} \label{remark} Assume that $k$ is perfect. Let $\mu_{2}$ denote the group scheme of the $2$nd roots of unity over $k.$ For $n \geq 10,$ the Hodge cohomology of $BG$ is no longer a polynomial ring. To determine the relations that hold in $H^{*}_{\textup{H}}(BG/k),$ we will restrict cohomology classes to a certain subgroup of $G$ considered in \cite[Section 12]{Tot}. Let $r= \lfloor n/2 \rfloor$ and let $T \cong \mathbb{G}_{m}^{r}$ denote a split maximal torus of $G.$ Assume that $n \not\equiv 2 \mod{4}$ so that the Weyl group $W$ of $G$ contains $-1,$ acting by inversion on $T.$ Then $-1$ acts by the identity on $T[2] \cong \mu_{2}^{r}$ and $G$ contains a subgroup $Q \cong \mu_{2}^{r} \times \mathbb{Z}/2.$ Under the double cover $G \to SO(n)_{k},$ the image of $Q$ is isomorphic to $K \cong \mu_{2}^{r-1} \times \mathbb{Z}/2$ and $Q \to K$ is a split surjection. We will need to know the Hodge cohomology rings of the classifying stacks of these groups. From \cite[Proposition 10.1]{Tot}, $$H^{*}_{\textup{H}}(B\mu_{2}/k)\cong k[t]\langle v \rangle$$ where $t \in H^{1}(B\mu_{2}, \Omega^{1})$ and $v \in H^{0}(B\mu_{2}, \Omega^{1}).$ From \cite[Lemma 10.2]{Tot}, $$H^{*}_{\textup{H}}((B\mathbb{Z}/2)/k)\cong k[s]$$ where $s \in H^{1}(B\mathbb{Z}/2, \Omega^{0}).$ The K\"{u}nneth formula \cite[Proposition 5.1]{Tot} then lets us calculate the Hodge cohomology ring of $B\mu_{2}^{i} \times B(\mathbb{Z}/2)^{j}$ for any $i,j \geq 0.$ Fix $i,j > 0$ and let $\textup{rad} \subset H^{*}_{\textup{H}}((B\mu_{2}^{i} \times B(\mathbb{Z}/2)^{j})/k)$ denote the ideal generated by nilpotent elements. Then $$H^{*}_{\textup{H}}((B\mu_{2}^{i} \times B(\mathbb{Z}/2)^{j})/k)/\textup{rad} \cong k[t_{1}, \ldots, t_{i}, s_{1}, \ldots, s_{j}]$$ where $t_{l} \in H^{1}(B\mu_{2}^{i} \times B(\mathbb{Z}/2)^{j}, \Omega^{1})$ for all $l$ and $s_{l} \in H^{1}(B\mu_{2}^{i} \times B(\mathbb{Z}/2)^{j}, \Omega^{0})$ for all $l.$

\end{remark}
\begin{theorem}
Let $n=10.$ The Hodge spectral sequence for $BG$ degenerates and $$H^{*}_{\textup{dR}}(BG/k)\cong H_{\textup{H}}^{*}(BG/k)\cong k[y_{4}, y_{6}, y_{7}, y_{8}, y_{10},y_{32}]/(y_{7}y_{10})$$ where $|y_{i}|=i$ for $i=4,6,7,8,10,32.$ 
\end{theorem}

\begin{proof}
By the K\"{u}nneth formula \cite[Proposition 5.1]{Tot}, we may assume that $k=\mathbb{F}_{2}$ so that Remark \ref{remark} applies. From \cite[Section 12]{Tot}, $$O(\frak{g})^{G} \cong k[y_{4}, y_{6}, y_{8},y_{10},y_{32}]$$ where $|y_{i}|=i$ in $H_{\textup{H}}^{*}(BG/k)$, viewing  $O(\frak{g})^{G}$ as a subring of $H_{\textup{H}}^{*}(BG/k).$ Consider the spectral sequence \begin{equation} \label{SSSpin10} E_{2}^{i,j}=H_{\textup{H}}^{i}(BG/k)\otimes H_{\textup{H}}^{j}((G/P)/k)\Rightarrow H_{\textup{H}}^{i+j}(BL/k) \end{equation} from Proposition \ref{Serre SS}. From Proposition \ref{cohflagvar} and Corollary \ref{cohofBL}, $$H_{\textup{H}}^{*}((G/P)/k) \cong k[e_{1},e_{2},e_{3},e_{4}]/(e_{i}^{2}=e_{2i})$$ and $$H_{\textup{H}}^{*}(BL/k) \cong k[A, c_{2}, c_{3},c_{4},c_{5}].$$ Calculations similar to those performed in the proof of Proposition \ref{spin7prop} show that all differentials in \ref{SSSpin10} are zero on $e_{1}$ and $e_{3}\in E^{0,6}_{2}$ is transgressive with $0 \neq d_{7}(e_{3})=y_{7} \in H^{4}(BG, \Omega^{3}).$ We have $H^{m}_{\textup{H}}(BG/k) \cong H^{m}_{\textup{H}}(B\textup{Spin}(9)_{k}/k)$ for $m < 10.$

\indent Let $F_{*}$ be the spectral sequence with $E_{2}$ page given by $F_{2}=\Delta(e_{1},e_{2},e_{4})$ where $e_{i}$ has bidegree $(0,2i)$ for all $i.$ As all differentials in the spectral sequence \ref{SSSpin10} are $0$ on $e_{1},$ there exists a map of spectral sequence $F_{*} \to E_{*}$ taking $e_{i}$ to $e_{i}$ for $i=1,2,4.$ Fix a variable $y.$ Let $H_{*}$ denote the spectral sequence with $E_{2}$ page $H_{2}=\Delta(e_{3}) \otimes k[y]$ where $e_{3}$ has bidegree $(0,6),$ $y$ has bidegree $(7,0),$ and $e_{3}$ is transgressive with $d_{7}(e_{3}y^{i})=y^{i+1}$ for all $i.$ As $e_{3}$ is transgessive in \ref{SSSpin10} with $d_{7}(e_{3})=y_{7},$ there exists a map of spectral sequences $H_{*} \to E_{*}$ taking $e_{3}$ to $e_{3}$ and $y$ to $y_{7}.$ Elements in the ring of $G$-invariants $k[y_{4}, y_{6}, y_{8},y_{10},y_{32}]$ are permanent cycles in \ref{SSSpin10}. Tensoring maps of spectral sequences, we get a map \begin{equation} \label{alphaspin10}\alpha:I_{*} \coloneqq F_{*} \otimes H_{*} \otimes k[y_{4}, y_{6}, y_{8},y_{10},y_{32}] \to E_{*} \end{equation} which induces an isomorphism on the $0$th columns of the $E_{2}$ pages.

\indent Let $n$ be even. The $k$-dimension of $H_{\textrm{H}}^{n}(BL/k)$ is equal to the cardinality of the set $$S_{n}=\{(a,b,c,d,e) \in \mathbb{Z}_{\geq 0}^{5}: 2a+4b+6c+8d+10e=n\}.$$ For $0\leq i \leq 15$, set $V_{i,n} \coloneqq H^{(n-2i)/2}(BG, \Omega^{(n-2i)/2}).$ For $0\leq i \leq 15$, $\textup{dim}_{k} V_{i,n}$ is equal to the cardinality of the set $S_{i,n}=\{(a,b,c,d,e) \in \mathbb{Z}_{\geq 0}^{5}: 4a+6b+8c+10d+32e=n-2i\}.$ As all differentials in \ref{SSSpin10} are $0$ on $e_{1} \in H^{2}_{\textup{H}}((G/P)/k),$ $$V_{i,n} \cong V_{i,n} \otimes k\cdot e_{1}^{i} \subseteq E^{n-2i,2i}_{\infty}$$ for $0\leq i \leq 7$. Hence, the map $\alpha$ \ref{alphaspin10} induces injections on all $E_{\infty}$ terms. For $n$ odd, $\alpha$ induces isomorphisms $0=I^{n-i,i}_{\infty} \cong E^{n-i,i}_{\infty}=0$ for all $i$ since $H^{*}_{\textup{H}}(BL/k)$ is concentrated in even degrees.

\indent Define a bijection $f_{n}:S_{n} \to \bigcup\limits_{i=0}^{15} S_{i,n}$ by 
$f_{n}(a,b,c,d,e)=(b,c,d,e,(a-i)/16)\in S_{i,n}$ for $a \equiv i \mod(16).$ Then \begin{equation} \label{dim of BL} \textup{dim}_{k} H_{\textrm{H}}^{n}(BL/k)=|S_{n}|=\sum_{i=0}^{15}|S_{i,n}|=\sum_{i=0}^{15}\textup{dim}_{k}V_{i,n}.\end{equation} Now assume that $n \leq 14.$ Then $f_{n}$ gives a bijection $$S_{n} \to \bigcup\limits_{i=0}^{7}S_{i,n}.$$ As $$\textup{dim}_{k} H_{\textrm{H}}^{n}(BL/k) \geq \sum_{i=0}^{7} E^{n-2i,2i}_{\infty}$$ and $V_{i,n} \subseteq E^{n-2i,2i}_{\infty}$ for $0\leq i \leq 7$, it follows that $V_{i,n} \cong E^{n-2i,2i}_{\infty}$ for $0\leq i \leq 7$ and $E^{n-2i,2i}_{\infty}=0$ for $i \geq 8.$ As $\alpha$ induces injections on all $E_{\infty}$ terms, Theorem \ref{Zeeman} implies that $\alpha$ \ref{alphaspin10} induces an isomorphism $I_{2}^{n,0} \to E_{2}^{n,0}$ for $n<16.$

\indent Now we consider the filtration on $H_{\textrm{H}}^{16}(BL/k)$ given by \ref{SSSpin10}. From the bijection $f_{16}$ defined in the previous paragraph, we have $$\textup{dim}_{k} H_{\textrm{H}}^{16}(BL/k)=1+\sum_{i=0}^{7}|S_{i,n}|=1+\sum_{i=0}^{7}\textup{dim}_{k} V_{i,n} \otimes k\cdot e_{1}^{i}.$$ As all differentials are $0$ on $e_{1}$ and $\alpha$ induces isomorphisms on $0$th row terms of the $E_{2}$ pages in degrees less than $16,$ we must then have $$E_{\infty}^{10,6} \cong (H^{10}_{\textup{H}}(BG/k) \otimes k \cdot e_{1}^{3}) \oplus (k \cdot z \otimes k \cdot e_{3})$$ for some $0 \neq z \in H^{10}_{\textup{H}}(BG/k).$ Hence, $y_{7}z=0$ in $H^{*}_{\textup{H}}(BG/k).$ Write $z=ay_{4}y_{6}+by_{10}$ for some $a, b \in k.$

\indent We now show that $a=0$ by restricting $y_{7}z=0$ to the Hodge cohomology of the classifying stack of the subgroup $\textup{Spin}(8)_{k}$ of $G.$ Under the isomorphism $$H^{*}_{\textup{H}}(B\textup{Spin}(8)_{k}/k) \cong k[y_{4},y_{6},y_{7},y_{8},y_{16}]$$ of Theorem \ref{theoremspin10}, pullback from $H^{*}_{\textup{H}}(BG/k)$ to $H^{*}_{\textup{H}}(B\textup{Spin}(8)_{k}/k)$ maps $y_{4} ,y_{6}, y_{10} \in H^{*}_{\textup{H}}(BG/k)$ to $y_{4},y_{6},$ and $0$ respectively in $H^{*}_{\textup{H}}(B\textup{Spin}(8)_{k}/k).$ Hence, to show that $a=0,$ it suffices to show that $y_{7} \in H^{*}_{\textup{H}}(BG/k)$ restricts to $y_{7} \in H^{*}_{\textup{H}}(B\textup{Spin}(8)_{k}/k).$ From the isomorphism $$H^{*}_{\textup{H}}(BSO(m)_{k}/k)\cong k[u_{2}, \ldots, u_{m}]$$ of \cite[Theorem 11.1]{Tot} for $m\geq0,$ the class $u_{7} \in H^{7}_{\textup{H}}(BSO(10)_{k}/k)$ restricts to $u_{7} \in H^{7}_{\textup{H}}(BSO(8)_{k}/k).$ Thus, we are reduced to showing that $u_{7} \in H^{7}_{\textup{H}}(BSO(8)_{k}/k)$ pulls back to a non-zero multiple of $y_{7} \in H^{*}_{\textup{H}}(B\textup{Spin}(8)_{k}/k).$ 

\indent Consider the subgroups $\mu_{2}^{4} \times \mathbb{Z}/2 \cong Q \subseteq \textup{Spin}(8)_{k}$ and $\mu_{2}^{3} \times  \mathbb{Z}/2 \cong K \subseteq SO(8)_{k}$ defined in Remark \ref{remark}. As the morphism $Q \to K$ is split surjective, if we can show that $u_{7}$ restricts to a nonzero class in $H^{*}_{\textup{H}}(BK/k),$ then $u_{7}$ would restrict to a nonzero class in $H^{7}_{\textup{H}}(B\textup{Spin}(8)_{k}/k).$ From the inclusion $O(2)^{4}_{k} \subset O(8)_{k},$ $O(8)_{k}$ contains a subgroup of the form $\mu_{2}^{4} \times (\mathbb{Z}/2)^{4}.$ As $SO(8)_{k}$ is the kernel of the Dickson determinant $O(8)_{k}\to \mathbb{Z}/2,$ it follows that $SO(8)_{k}$ contains a subgroup $H\cong \mu_{2}^{4} \times (\mathbb{Z}/2)^{3}.$ Write $$H^{*}_{\textup{H}}(BH/k)/\textup{rad}\cong k[t_{1}, \ldots, t_{4}, s_{1}, \ldots, s_{4}]/(s_{1}+s_{2}+s_{3}+s_{4})$$ using Remark \ref{remark}. From the proof of \cite[Lemma 11.4]{Tot}, the pullback of $u_{7}$ to $H^{*}_{\textup{H}}(BH/k)/\textup{rad}$ followed by pullback to $H^{*}_{\textup{H}}(BK/k)/\textup{rad} \cong k[t_{1}, \ldots, t_{4},s]/(t_{1}+\cdots+t_{4})$ is given by

$$ u_{7} \mapsto \sum_{j=1}^{3} s_{j}(t_{j}+t_{4}) \sum_{\substack{1 \leq i_{1} < i_{2} \leq 3 \\ i_{1}, i_{2} \neq j}} t_{i_{1}}t_{i_{2}} \mapsto \sum_{j=1}^{3} s(t_{j}+t_{4}) \sum_{\substack{1 \leq i_{1} < i_{2} \leq 3 \\ i_{1}, i_{2} \neq j}} t_{i_{1}}t_{i_{2}}$$
\\
$$=s\sum_{1 \leq i_{1}<i_{2}\leq 3} (t_{i_{1}}+t_{i_{2}})t_{i_{1}}t_{i_{2}} \neq 0.$$
 Thus, $u_{7} \in H^{7}_{\textup{H}}(BSO(8)_{k}/k)$ pulls back to a nonzero multiple of $y_{7} \in H^{7}_{\textup{H}}(B\textup{Spin}(8)_{k}/k)$ which implies that $y_{7}y_{10}=0$ in $H^{*}_{\textup{H}}(BG/k).$

\[
\begin{tikzcd}
u_{7} \in H^{7}_{\textup{H}}(BSO(10)_{k}/k) \arrow[d,""] \arrow[r, ""] & y_{7} \in H^{7}_{\textup{H}}(BG/k) \arrow[d,""] \\
u_{7} \in H^{7}_{\textup{H}}(BSO(8)_{k}/k) \arrow[d,""] \arrow[r, ""] & y_{7} \in H^{7}_{\textup{H}}(B\textup{Spin}(8)_{k}/k) \arrow[d,""] \\
\sum_{j=1}^{3} s(t_{j}+t_{4}) \sum_{\substack{1 \leq i_{1} < i_{2} \leq 3 \\ i_{1}, i_{2} \neq j}} t_{i_{1}}t_{i_{2}} \in H^{7}_{\textup{H}}(BK/k) \arrow[r,"\neq 0"] &  H^{7}_{\textup{H}}(BQ/k)
\end{tikzcd}
\]

\indent Using the relation $y_{7}y_{10}=0,$ we now modify the spectral sequence $I_{*}$ defined above to define a new spectral sequence $J_{*}$ that better approximates (and will actually be isomorphic to) the spectral sequence \ref{SSSpin10}. Let $$(yy_{10}) \coloneqq F_{2} \otimes (\Delta(e_{3})\otimes yk[y])\otimes y_{10}k[y_{4},y_{6},y_{8},y_{10},y_{32}].$$ Define the $E_{2}$ page of $J_{*}$ by $J_{2}=I_{2}/(yy_{10})$. Define the differentials $d'_{m}$ of $J_{*}$ so that $I_{2} \to J_{2}$ induces a map $I_{*} \to J_{*}$ of cohomological spectral sequences of $k$-vector spaces and $d'_{m}=0$ for $m>7.$ This means that $d_{7}'(f \otimes e_{3} \otimes y_{10}g)= f \otimes y \otimes y_{10}g=0$ and $d_{m}'(f \otimes e_{3} \otimes y_{10}g)=0$ for $m>7$, $f \in F_{2},$ and $g \in k[y_{4},y_{6},y_{8},y_{10},y_{32}].$ The $E_{\infty}$ page of $J_{*}$ is given by $$J_{\infty} \cong (F_{2} \otimes k[y_{4},y_{6},y_{8},y_{10},y_{32}]) \oplus (F_{2} \otimes e_{3} \otimes y_{10}k[y_{4},y_{6},y_{8},y_{10},y_{32}]).$$

 \indent As $y_{7}y_{10}=0$ in $H^{*}_{\textup{H}}(BG/k)$, $\alpha$ induces a map $\alpha ':J_{*} \to E_{*}$ of spectral sequences. To finish the calculation, we will show that $\alpha '$ induces an isomorphism on $E_{\infty}$ terms so that Theorem \ref{Zeeman} will apply. For $n$ odd, $E^{n-i,i}_{\infty}=0$ for all $i$ since $H^{*}_{\textup{H}}(BL/k)$ is concentrated in even degrees. Now assume that $n$ is even. For $0 \leq i \leq 7,$ $$V_{i,n} \cong H^{(n-2i)/2}(BG, \Omega^{(n-2i)/2}) \otimes e_{1}^{i} \subseteq E^{n-2i,2i}_{\infty}.$$ For $8 \leq i \leq 15$, $$V_{i,n}\cong y_{10}H^{(n-2i)/2}(BG, \Omega^{(n-2i)/2}) \otimes e_{1}^{i-8}e_{3} \subseteq E^{n-2i+10,2i-10}_{\infty}.$$ Hence, from the description of the $E_{\infty}$ terms of $J_{*}$ given above, it follows that $\alpha '$ induces an injection $J_{\infty}^{n-2i,2i} \to E_{\infty}^{n-2i,2i}$ for all $i.$ Equation \ref{dim of BL} then implies that $J_{\infty}^{n-2i,2i} \cong E_{\infty}^{n-2i,2i}$ for all $i.$ 
 
 \indent Thus, $\alpha '$ induces an isomorphism on $E_{\infty}$ pages and an isomorphism on the $0$th columns of the $E_{2}$ pages of the $2$ spectral sequences. Theorem \ref{Zeeman} then implies that $$H_{\textup{H}}^{*}(BG/k)\cong k[y_{4}, y_{6}, y_{7}, y_{8}, y_{10},y_{32}]/(y_{7}y_{10}).$$ From Proposition \ref{HodgeSSG}, the Hodge spectral sequence for $BG$ degenerates.

\end{proof}

\begin{corollary}
Let $G$ be a $k$-form of $\textup{Spin}(10).$ Then $$H^{*}_{\textup{H}}(BG/k)\cong k[y_{4},y_{6},y_{7},y_{8},y_{10},y_{32}]/(y_{7}y_{10})$$ where $|y_{i}|=i$ for all $i.$ 
\end{corollary}

\begin{theorem}
 Let $n=11.$ The Hodge spectral sequence for $BG$ degenerates and $$H^{*}_{\textup{dR}}(BG/k)\cong H_{\textup{H}}^{*}(BG/k)\cong k[y_{4}, y_{6}, y_{7}, y_{8}, y_{10},y_{11},y_{32}]/(y_{7}y_{10}+y_{6}y_{11})$$ where $|y_{i}|=i$ for $i=4,6,7,8,10,11,32.$ 
\end{theorem}

\begin{proof}

\indent We may assume that $k=\mathbb{F}_{2}$ so that Remark \ref{remark} applies. From \cite[Section 12]{Tot}, $$O(\frak{g})^{G} \cong k[y_{4}, y_{6}, y_{8},y_{10},y_{32}]$$ where $|y_{i}|=i$ in $H_{\textup{H}}^{*}(BG/k)$, viewing  $O(\frak{g})^{G}$ as a subring of $H_{\textup{H}}^{*}(BG/k).$ Consider the spectral sequence \begin{equation} \label{SSSpin11} E_{2}^{i,j}=H_{\textup{H}}^{i}(BG/k)\otimes H_{\textup{H}}^{j}((G/P)/k)\Rightarrow H_{\textup{H}}^{i+j}(BL/k) \end{equation} from Proposition \ref{Serre SS}. From Proposition \ref{cohflagvar} and Corollary \ref{cohofBL}, $$H_{\textup{H}}^{*}((G/P)/k) \cong k[e_{1},e_{2},e_{3},e_{4},e_{5}]/(e_{i}^{2}=e_{2i})$$ and $$H_{\textup{H}}^{*}(BL/k) \cong k[A, c_{2}, c_{3},c_{4},c_{5}].$$ 
\indent Using \cite[Theorem 11.1]{Tot}, write $H^{*}_{\textup{H}}(BSO(11)_{k}/k)\cong k[u_{2}, \ldots, u_{11}].$ From the inclusions $O(2)^{5}_{k} \subset O(10)_{k} \subset SO(11)_{k}$, $SO(11)_{k}$ contains a subgroup $H \cong \mu_{2}^{5} \times (\mathbb{Z}/2)^{5}.$ Write $H^{*}_{\textup{H}}(BH/k)/\textup{rad} \cong k[t_{1},\ldots,t_{5},s_{1},\ldots,s_{5}]$ as described in Remark \ref{remark}. Under the pullback map $H^{*}_{\textup{H}}(BSO(11)_{k}/k) \to H^{*}_{\textup{H}}(BH/k)/\textup{rad},$   $u_{2m}$ pulls back to the $m$th elementary symmetric polynomial $$\sum_{1 \leq i_{1}<\cdots<i_{m} \leq 5}t_{i_{1}}\cdots t_{i_{m}}$$ and $u_{2m+1}$ pulls back to $$\sum_{j=1}^{5} s_{j}\sum_{\substack{1 \leq i_{1}<\cdots<i_{m} \leq 5\\ \textup{one equal to j}}}t_{i_{1}}\cdots t_{i_{m}}$$ for $1 \leq m \leq 5$ \cite[Lemma 11.4]{Tot}.

\indent Let $Q\cong (\mu_{2}^{5} \times \mathbb{Z}/2) \subset G$ and $K \cong (\mu_{2}^{4} \times \mathbb{Z}/2) \subset SO(11)_{k}$ be the subgroups described in Remark \ref{remark}. Write $H^{*}_{\textup{H}}(BK/k)/\textup{rad} \cong k[t_{1},\ldots,t_{5},s]/(t_{1}+\cdots+t_{5}).$ Under the pullback map $H^{*}_{\textup{H}}(BSO(11)_{k}/k) \to H^{*}_{\textup{H}}(BK/k)/\textup{rad},$ $u_{7}$ maps to $su_{6} \neq 0$ and $u_{11}$ maps to $su_{10}\neq 0.$ As $Q \to K$ is split, it follows that $u_{7},u_{11}$ restrict to nonzero classes $y_{7} \in H^{7}_{\textup{H}}(BG/k)$ and $y_{11} \in H^{11}_{\textup{H}}(BG/k).$ Also, $y_{4}y_{7}$ and $y_{11}$ are linearly independent in $H^{11}_{\textup{H}}(BG/k).$

\indent Returning to the spectral sequence \ref{SSSpin11}, calculations similar to those performed in the proof of Proposition \ref{spin7prop} show that all differentials in \ref{SSSpin11} are zero on $e_{1}$ and $e_{3}\in E^{0,6}_{2}$ is transgressive with $0 \neq d_{7}(e_{3})=y_{7} \in H^{4}(BG, \Omega^{3}).$ We have $H^{m}_{\textup{H}}(BG/k) \cong H^{m}_{\textup{H}}(B\textup{Spin}(10)_{k}/k)$ for $m \leq 10.$

\indent Let $F_{*}$ be the spectral sequence with $E_{2}$ page given by $\Delta(e_{1},e_{2},e_{4})$ with $e_{i}$ of bidegree $(0,2i)$ for $i=1,2,4.$ Fix a variable $y$ and let $H_{*}$ be the spectral sequence with $H_{2}=\Delta(e_{3}) \otimes k[y]$ where $e_{3}$ is of bidegree $(0,6),$ $y$ is of bidegree $(7,0),$ and $e_{3}$ is transgressive with $d_{7}(e_{3}y^{i})=y^{i+1}$ for all $i.$ There exists a map of spectral sequence $$\alpha:I_{*}\coloneqq F_{*} \otimes H_{*} \otimes k[y_{4}, y_{6}, y_{8},y_{10},y_{32}] \to E_{*}$$ taking $e_{i}$ to $e_{i}$ for $i=1,2,3,4$ and taking $y$ to $y_{7}.$ The $E_{\infty}$ page of $I_{*}$ is given by $I_{\infty} \cong F_{2} \otimes k[y_{4}, y_{6}, y_{8},y_{10},y_{32}]$ and $\alpha$ induces an injection $I_{\infty}^{i,j} \to E_{\infty}^{i,j}$ for all $i,j$ with $i+j \leq 17.$ For $n$ odd, $\alpha$ induces an isomorphism $0=I^{n-i,i}_{\infty} \cong E^{n-i,i}_{\infty}=0$ for all $i$ since $H^{*}_{\textup{H}}(BL/k)$ is concentrated in even degrees.

\indent Let $n$ be even. The $k$-dimension of $H_{\textrm{H}}^{n}(BL/k)$ is equal to the cardinality of the set $$S_{n}=\{(a,b,c,d,e) \in \mathbb{Z}_{\geq 0}^{5}: 2a+4b+6c+8d+10e=n\}.$$ For $0\leq i \leq 15$, set $V_{i,n} \coloneqq H^{(n-2i)/2}(BG, \Omega^{(n-2i)/2}).$ For $0\leq i \leq 15$, $\textup{dim}_{k} V_{i,n}$ is equal to the cardinality of the set $S_{i,n}=\{(a,b,c,d,e) \in \mathbb{Z}_{\geq 0}^{5}: 4a+6b+8c+10d+32e=n-2i\}.$ As all differentials in \ref{SSSpin11} are $0$ on $e_{1} \in H^{2}_{\textup{H}}((G/P)/k),$ $$V_{i,n} \cong V_{i,n} \otimes k\cdot e_{1}^{i} \subseteq E^{n-2i,2i}_{\infty}$$ for $0\leq i \leq 7$ and $n \leq 16.$

\indent Define a bijection $f_{n}:S_{n} \to \bigcup\limits_{i=0}^{15} S_{i,n}$ by 
$f_{n}(a,b,c,d,e)=(b,c,d,e,(a-i)/16)\in S_{i,n}$ for $a \equiv i \mod(16).$ Then \begin{equation} \label{dim of BLSpin11} \textup{dim}_{k} H_{\textrm{H}}^{n}(BL/k)=|S_{n}|=\sum_{i=0}^{15}|S_{i,n}|=\sum_{i=0}^{15}\textup{dim}_{k}V_{i,n}.\end{equation} Now assume that $n \leq 14.$ Then $f_{n}$ gives a bijection $$S_{n} \to \bigcup\limits_{i=0}^{7}S_{i,n}.$$ As $$\textup{dim}_{k} H_{\textrm{H}}^{n}(BL/k) \geq \sum_{i=0}^{7} E^{n-2i,2i}_{\infty}$$ and $V_{i,n} \subseteq E^{n-2i,2i}_{\infty}$ for $0\leq i \leq 7$, it follows that $V_{i,n} \cong E^{n-2i,2i}_{\infty}$ for $0\leq i \leq 7$ and $E^{n-2i,2i}_{\infty}=0$ for $i \geq 8.$ In particular, $E^{0,10}_{\infty} \cong k\cdot e_{1}^{5}.$ As mentioned above, we have $H^{m}_{\textup{H}}(BG/k)=0$ for $m=3,5,9.$ After adding a $k$-multiple of $e_{3}e_{1}^{2}$ to $e_{5},$ we can assume that $d_{7}(e_{5})=0.$ Then the isomorphism $E^{0,10}_{\infty} \cong k\cdot e_{1}^{5}$ implies that $d_{11}(e_{5}) \neq 0.$ Hence, $e_{5}$ is transgressive in \ref{SSSpin11} and $y_{11}\in H^{6}(BG, \Omega^{5})$ is a lifting of $d_{11}(e_{5})$ to $E_{2}^{11,0}.$

\indent Fix a variable $x.$ Let $J_{*}$ denote the spectral sequence with $E_{2}$ page $J_{2}=\Delta(e_{5})\otimes k[x]$ where $e_{5}$ has bidegree $(0,10),$ $x$ has bidegree $(11,0),$ and $e_{5}$ is transgressive with $d_{11}(e_{5}x^{i})=x^{i+1}$ for all $i.$ 

\[
\begin{tikzcd}
k\cdot e_{5} \arrow[dr, "d_{11}"] &  k \cdot e_{5}x \arrow[dr, "d_{11}"] & \cdots  \\
0 &  k \cdot x & k \cdot x^2 & \cdots
\end{tikzcd}
\]

As $e_{5}$ is transgressive in \ref{SSSpin11}, there exists a map of spectral sequences $J_{*} \to E_{*}$ taking $e_{5}$ to $e_{5}$ and $x$ to $y_{11}.$ Tensoring with the map $\alpha$ defined above, we get a map $$\alpha ': K_{*}\coloneqq I_{*} \otimes J_{*} \to E_{*}$$ which induces an isomorphism on the $0$th columns of the $E_{2}$ pages. The $E_{\infty}$ page of $K_{*}$ is given by $$K_{\infty} \cong I_{\infty} \cong F_{2} \otimes k[y_{4}, y_{6}, y_{8},y_{10},y_{32}].$$ As mentioned above, $\alpha$ and hence $\alpha '$ induce isomorphisms on $E^{n-i,i}_{\infty}$ terms for $n<16$ and injections on all $E_{\infty}$ terms on or below the line $i+j=17.$ Theorem \ref{Zeeman} implies that $\alpha '$ induces an isomorphism $K^{n,0}_{2} \to E^{n,0}_{2}$ for $n<16.$

\indent Next, we consider the filtration on $H^{16}_{\textup{H}}(BL/k)$ given by \ref{SSSpin11}. From \ref{dim of BLSpin11}, $$\textup{dim}_{k} H_{\textrm{H}}^{16}(BL/k)=1+\sum_{i=0}^{7}|S_{i,16}|=1+\sum_{i=0}^{7}\textup{dim}_{k} V_{i,16} \otimes k\cdot e_{1}^{i}.$$ We must then have either $d_{7}(e_{3}f)=y_{7}f=0\in H^{17}_{\textup{H}}(BG/k)$ for some $0 \neq f \in H^{10}_{\textup{H}}(BG/k)$ or $d_{11}(e_{5}g)=y_{11}g=d_{7}(e_{3})h=y_{7}h \in H^{17}_{\textup{H}}(BG/k)$ for some $0 \neq g \in H^{6}_{\textup{H}}(BG/k)$ and $h \in H^{10}_{\textup{H}}(BG/k).$ Let $a,b,c \in k$, not all zero, such that $ay_{11}y_{6}+by_{7}y_{10}+cy_{7}y_{4}y_{6}=0\in H^{17}_{\textup{H}}(BG/k).$

\indent The class $au_{11}u_{6}+bu_{7}u_{10}+cu_{7}u_{4}u_{6} \in H^{17}_{\textup{H}}(BSO(11)_{k}/k)$ pulls back to $ay_{11}y_{6}+by_{7}y_{10}+cy_{7}y_{4}y_{6}=0 \in H^{17}_{\textup{H}}(BG/k).$ Under the pullback map $$H^{*}_{\textup{H}}(BSO(11)_{k}/k) \to H^{*}_{\textup{H}}(BK/k)/\textup{rad} \cong  k[t_{1},\ldots,t_{5},s]/(t_{1}+\cdots+t_{5}),$$ $au_{11}u_{6}+bu_{7}u_{10}+cu_{7}u_{4}u_{6}$ maps to $asu_{10}u_{6}+bsu_{6}u_{10}+csu_{6}u_{4}u_{6}$, which equals $0$ since $Q \to K$ is split. Then $c=0$ and $a=b.$  Thus, the relation $y_{7}y_{10}+y_{6}y_{11}=0$ holds in $H^{*}_{\textup{H}}(BG/k)$ and $E^{6,10}_{\infty} \cong (k \cdot y_{6} \otimes e_{5}) \oplus (k \cdot y_{6} \otimes e_{1}^{5}).$ We now use the relation $y_{7}y_{10}+y_{6}y_{11}=0$ to define a new spectral sequence $L_{*}$ from $K_{*}.$

\indent Let $(y_{6}x+yy_{10}) \subset K_{2}$ denote the ideal generated by $y_{6}x+yy_{10}$ and let $L_{2} \coloneqq K_{2}/(y_{6}x+yy_{10}).$ Define the differentials $d'_{m}$ of $L_{*}$ so that $K_{2} \to L_{2}$ induces a map of spectral sequences $K_{*} \to L_{*}$ and $d'_{m}=0$ for $m>11.$ Then $\alpha ':K_{*} \to E_{*}$ induces a map of spectral sequences $\alpha '':L_{*} \to E_{*}.$ The $E_{\infty}$ page of $L_{*}$ is given by $$L_{\infty} \cong (F_{2} \otimes k[y_{4},y_{6},y_{8},y_{10},y_{32}]) \oplus (F_{2} \otimes y_{6}k[y_{4},y_{6},y_{8},y_{10},y_{32}] \otimes e_{5}).$$

\indent We now show by induction that $\alpha ''$ induces an isomorphism $L_{2}^{n,0} \to E_{2}^{n,0}$ for all $n.$ For $n<16,$ we have shown that $L_{2}^{n,0} \cong E_{2}^{n,0}.$ Now let $n\geq 16$ and assume that $\alpha ''$ induces an isomorphism $L_{2}^{m,0} \to E_{2}^{m,0}$ for all $m<n.$ First, suppose that $n$ is even. As $L_{2}^{m,0} \cong E_{2}^{m,0}$ for $m<n,$ $y_{7}g \neq 0 \in H^{*}_{\textup{H}}(BG/k)$ for all $0 \neq g \in H^{*}_{\textup{H}}(BG/k)$ with $|g|<n-7.$ Hence, for any $0 \neq g \in H^{m}_{\textup{H}}(BG/k)$ with $|g|=m<n-7,$ $g \otimes e_{3}e_{5} \in E_{2}^{m,16}\cong E_{7}^{m,16}$ is not in the kernel of the differential $d_{7}:E_{7}^{m,16} \to E_{7}^{m+7,10} \cong E_{2}^{m+7,10}.$ As $y_{7} \in H^{4}(BG,\Omega^{3})$ and $y_{11} \in H^{6}(BG,\Omega^{5}),$ $y_{7}z, y_{11}z \not \in \oplus_{i} H^{i}(BG,\Omega^{i})$ for all $z \in H^{*}_{\textup{H}}(BG/k).$ It follows that $\alpha ''$ induces an injection $L^{i,j}_{\infty} \to E^{i,j}_{\infty} $ for all $i,j$ with $m=i+j\leq n:$ $$L_{\infty}^{m-2i,2i} \cong V_{i,m} \otimes e_{1}^{i} \subseteq E_{\infty}^{m-2i,2i}$$ for $0 \leq i \leq 4,$ $$L_{\infty}^{m-2i,2i} \cong (V_{i,m} \otimes e_{1}^{i}) \oplus (y_{6}V_{i+3,m}\otimes e_{1}^{i-5}e_{5}) \subseteq E_{\infty}^{m-2i,2i}$$ for $5 \leq i \leq 7,$ and $$L_{\infty}^{m-2i,2i} \cong y_{6}V_{i+3,m}\otimes e_{1}^{i-5}e_{5} \subseteq E_{\infty}^{m-2i,2i}$$ for $8 \leq i \leq 12.$ The equality in \ref{dim of BLSpin11} then implies that $\alpha''$ induces isomorphisms $L_{\infty}^{i,j} \to E_{\infty}^{i,j}$ for all $i,j$ with $i+j\leq n.$ As mentioned above, $\alpha ''$ induces isomorphisms $0=L^{n+1-i,i}_{\infty} \to E^{n+1-i,i}_{\infty}=0$ for all $i$ since $n+1$ is odd. Theorem \ref{Zeeman} then implies that $\alpha ''$ induces an isomorphism $L^{n,0}_{2} \cong E^{n,0}_{2}=H^{n}_{\textup{H}}(BG/k).$

\indent Now assume that $n$ is odd. We have $0=L_{\infty}^{i,j} \cong E_{\infty}^{i,j}=0$ for all $i,j$ with $i+j=n.$ An argument similar to the one used above for when $n$ is even shows that $\alpha ''$ induces injections $L^{i,j}_{\infty} \to E^{i,j}_{\infty} $ for all $i,j$ with $i+j\leq n+1.$ Equation \ref{dim of BLSpin11} then implies that $\alpha''$ induces isomorphisms $L_{\infty}^{i,j} \to E_{\infty}^{i,j}$ for all $i,j$ with $i+j\leq n+1.$ It follows that $\alpha ''$ induces an isomorphism $L^{n,0}_{2} \cong E^{n,0}_{2}=H^{n}_{\textup{H}}(BG/k)$ by an application of Theorem \ref{Zeeman}. Thus, by induction, we have obtained that the $0$th row of $L_{2}$ is isomorphic to the $0$th row of $E_{2}:$ $$H_{\textup{H}}^{*}(BG/k)\cong k[y_{4}, y_{6}, y_{7}, y_{8}, y_{10},y_{11},y_{32}]/(y_{7}y_{10}+y_{6}y_{11}).$$ The Hodge spectral sequence for $BG$ degenerates by Proposition \ref{HodgeSSG}.

\end{proof}

\begin{corollary}
Let $G$ be a $k$-form of $\textup{Spin}(11).$ Then $$H_{\textup{H}}^{*}(BG/k)\cong k[y_{4}, y_{6}, y_{7}, y_{8}, y_{10},y_{11},y_{32}]/(y_{7}y_{10}+y_{6}y_{11})$$ where $|y_{i}|=i$ for $i=4,6,7,8,10,11,32.$ 
\end{corollary}

\end{document}